\documentclass{amsart}
\usepackage{amsmath, amsfonts, amssymb}
\usepackage{stmaryrd,epsfig}
\usepackage{psfrag}

\newtheorem*{thm1}{Theorem~1}

\newtheorem{thm}{Theorem}[section]
\newtheorem{lemma}[thm]{Lemma}
\newtheorem{prop}[thm]{Proposition}
\newtheorem{cor}[thm]{Corollary}
\newtheorem{conj}[thm]{Conjecture}

\theoremstyle{definition}
\newtheorem{remark}[thm]{Remark}
\newtheorem{example}[thm]{Example}
\newtheorem{defn}[thm]{Definition}

\DeclareMathOperator{\OF}{OF}

\DeclareMathOperator{\Aut}{Aut}

\DeclareMathOperator{\SO}{SO}

\DeclareMathOperator{\Span}{Span}

\DeclareMathOperator{\QH}{QH}

\DeclareMathOperator{\GW}{GW}

\DeclareMathOperator{\codim}{codim}

\DeclareMathOperator{\supp}{supp}

\newcommand{\N}{{\mathbb N}}
\newcommand{\Z}{{\mathbb Z}}

\newcommand{\C}{{\mathbb C}}

\newcommand{\cO}{{\mathcal O}}

\newcommand{\op}{\text{op}}

\newcommand{\gw}[2]{\langle #1 \rangle^{\mbox{}}_{#2}}

\newcommand{\euler}[1]{\chi_{_{#1}}}

\newcommand{\pt}{\text{point}}

\newcommand{\ev}{\operatorname{ev}}

\newcommand{\wh}{\widehat}
\newcommand{\wb}{\overline}
\newcommand{\ov}{\overline}

\newcommand{\pic}[2]{\includegraphics[scale=#1]{#2}}

\newcommand{\ignore}[1]{}

\newcommand{\Mb}{\wb{\mathcal M}}

\newcommand{\La}{\Lambda}
\newcommand{\la}{\lambda}
\newcommand{\ssm}{\smallsetminus}

\def\al{\alpha}
\def\be{\beta}
\def\de{\delta}
\def\ga{\gamma}

\def\alv{{\alpha^\vee}}
\def\bev{{\beta^\vee}}
\def\dev{{\delta^\vee}}
\def\gav{{\gamma^\vee}}
\def\noin{\noindent}
\def\R{{\mathbb R}}

\DeclareMathOperator{\height}{height}

\begin{document}

\title{Curve neighborhoods of Schubert varieties}

\date{March 24, 2013}

\author{Anders~S.~Buch}
\address{Department of Mathematics, Rutgers University, 110
  Frelinghuysen Road, Piscataway, NJ 08854, USA}
\email{asbuch@math.rutgers.edu}

\author{Leonardo~C.~Mihalcea}
\address{460 McBryde Hall, Department of Mathematics, Virginia Tech,
  Blacksburg, VA 24061, USA}
\email{lmihalce@math.vt.edu}


\thanks{The first author was supported in part by NSF grants
  DMS-0906148 and DMS-1205351.}

\thanks{The second author was supported in part by NSA Young
  Investigator Award 12-0871-10.}

\begin{abstract}
  A previous result of the authors with Chaput and Perrin states that
  the union of all rational curves of fixed degree passing through a
  Schubert variety in a homogeneous space $G/P$ is again a Schubert
  variety.  In this paper we identify this Schubert variety explicitly
  in terms of the Hecke product of Weyl group elements.  We apply our
  result to give an explicit formula for any two-point Gromov-Witten
  invariant as well as a new proof of the quantum Chevalley formula
  and its equivariant generalization.  We also recover a formula for
  the minimal degree of a rational curve between two given points in a
  cominuscule variety.
\end{abstract}

\maketitle


\section{Introduction}

Let $X = G/P$ be a homogeneous space defined by a semisimple complex
Lie group $G$ and a parabolic subgroup $P$.  The quantum cohomology
ring of $X$ is closely related to the geometry of rational curves in
$X$ and has received much attention since the mid 1990's.  Given a
subvariety $\Omega \subset X$ and an effective degree $d \in H_2(X)$,
define the {\em curve neighborhood\/} $\Gamma_d(\Omega)$ to be the
closure of the union of all rational curves of degree $d$ in $X$ that
meet $\Omega$.  Recent developments suggest that this variety is a key
object in the study of the quantum $K$-theory ring of $X$.

The study of $\Gamma_d(\Omega)$ was initiated in the recent paper
\cite{buch.chaput.ea:finiteness} by Chaput, Perrin, and the authors.
Using that the Kontsevich moduli space of stable maps to $X$ is
irreducible it was proved that, if $\Omega$ is an irreducible
subvariety of $X$, then $\Gamma_d(\Omega)$ is also irreducible.  In
particular, if $\Omega$ is a Schubert variety in $X$, then so is
$\Gamma_d(\Omega)$.  The applications to quantum $K$-theory require a
precise description of this locus.  This was obtained in
\cite{buch.chaput.ea:finiteness} when $X$ is any cominuscule
homogeneous space.  For example, when $X$ is a Grassmann variety of
type A and $\Omega = \Omega_\la$ is a Schubert variety corresponding
to a Young diagram $\la$, the Young diagram associated to
$\Gamma_d(\Omega_\la)$ is obtained by removing the first $d$ rows and
columns from $\la$.  This operation on Young diagrams has appeared in
several references, possibly starting with
\cite{fulton.woodward:quantum}.  The main result of this paper is an
explicit combinatorial formula for the Weyl group element
corresponding to $\Gamma_d(\Omega)$ when $\Omega \subset X$ is a
Schubert variety in an arbitrary homogeneous space.  A description of
the curve neighborhood of a Richardson variety has been obtained in
\cite{li.mihalcea:k-theoretic} when $d$ is the degree of a line and
the Fano variety of lines of degree $d$ in $X$ is a homogeneous space.

Fix a maximal torus $T$ and a Borel subgroup $B$ such that $T \subset
B \subset P \subset G$.  Let $W$ be the Weyl group of $G$ and $W_P$
the Weyl group of $P$.  Each element $w \in W$ defines a Schubert
variety $X(w) = \ov{B w.P}$ in $X$ and an opposite Schubert variety
$Y(w) = \ov{B^\op w.P}$, where $B^\op$ is the opposite Borel subgroup.
If $w$ is the minimal representative for its coset in $W/W_P$, then we
have $\dim X(w) = \codim Y(w) = \ell(w)$.  Given a positive root $\al$
with $s_\al \notin W_P$, let $C_\al \subset X$ be the unique
$T$-stable curve that contains the $T$-fixed points $1.P$ and
$s_\al.P$.  The homology group $H_2(X) = H_2(X;\Z)$ can be identified
with the quotient $\Z\Delta^\vee / \Z\Delta_P^\vee$, where
$\Z\Delta^\vee$ is the coroot lattice of $G$ and $\Z\Delta_P^\vee$ is
the coroot lattice of $P$.  The degree $[C_\al] \in H_2(X)$ is equal
to the image of the coroot $\alv = \frac{2\al}{(\al,\al)}$ under this
identification.

Our description of $\Gamma_d(X(w))$ is formulated using the {\em Hecke
  product\/} on $W$, which by definition is the unique associative
monoid product such that, for any simple reflection $s_\be$ and $w \in
W$ we have
\[
w \cdot s_\be = \begin{cases}
  w s_\be & \text{if $\ell(w s_\be) > \ell(w)$;} \\
  w & \text{otherwise.}
\end{cases}
\]

\begin{thm1}
  Assume that $0 < d \in H_2(X)$, and let $\al$ be any maximal root
  with the property that $\alv \leq d$ as elements in $H_2(X)$.
  Then we have $\Gamma_d(X(w)) = \Gamma_{d-\alv}(X(w \cdot s_\al))$.
\end{thm1}

A root $\al$ will be called {\em $P$-cosmall\/} if $s_\al \notin W_P$
and $\al$ satisfies the condition of Theorem~1 for any positive degree
$d \in H_2(X)$.  This condition makes simultaneous use of the
partial orders of the root system and the dual root system of $G$,
which gives rise to some interesting combinatorics.

Let $z_d^P \in W$ denote the minimal representative for the curve
neighborhood of a point, i.e.\ $\Gamma_d(X(1)) = X(z_d^P)$.  Theorem~1
then implies that $\Gamma_d(X(w)) = X(w \cdot z_d^P)$.  Much of our
paper therefore focuses on the curve neighborhood of a point.

Curve neighborhoods are related to Fulton and Woodward's work
\cite{fulton.woodward:quantum} on determining the smallest degree of
the quantum parameter that appears in a product of Schubert classes in
the (small) quantum ring $\QH(X)$.  Let $\Mb_{0,n}(X,d)$ denote the
Kontsevich moduli space of $n$-pointed stable maps to $X$ of degree
$d$, with evaluation map $\ev = (\ev_1,\dots,\ev_n) : \Mb_{0,n}(X,d)
\to X^n$.  Then the curve neighborhood of $X(w)$ can be defined by
$\Gamma_d(X(w)) = \ev_1(\ev_2^{-1}(X(w)))$.  It is proved in
\cite{fulton.woodward:quantum} that the quantum product $[Y(u)] \star
[Y(w)]$ contains a term $q^{d'} [Y(v)]$ with $d' \leq d$ if and only
if the Gromov-Witten variety $\ev_1^{-1}(Y(u)) \cap \ev_2^{-1}(X(w_0
w))$ is not empty in $\Mb_{0,2}(X,d)$, where $w_0$ is the longest
element in $W$.  The later condition is equivalent to $Y(u) \cap
\Gamma_d(X(w_0 w)) \neq \emptyset$, which holds if and only if $u W_P
\leq w_0 w \cdot z_d^P W_P$ in the Bruhat order of $W/W_P$.

Define the Gromov-Witten variety $\GW_d(w) = \ev_2^{-1}(X(w)) \subset
\Mb_{0,2}(X,d)$ and consider the surjective map $\ev_1 : \GW_d(w) \to
\Gamma_d(X(w))$.  It was proved in \cite{buch.chaput.ea:finiteness}
that the general fibers of this map are unirational.  This implies
that the pushforward $(\ev_1)_*[\GW_d(w)] \in H^*_T(X)$ is equal to
$[X(w \cdot z_d^P)]$ whenever $\dim \GW_d(w) = \dim X(w \cdot z_d^P)$, and
otherwise $(\ev_1)_*[\GW_d(w)] = 0$.  It follows that any
(equivariant) two-point Gromov-Witten invariant of $X$ is given by
\begin{equation}\label{eqn:2gw}
  \begin{split}
    & I_d([Y(u)],[X(w)]) \ = \ 
    \int_{\Mb_{0,2}(X,d)} \ev_1^*[Y(u)] \cdot \ev_2^*[X(w)] \\
    & \ \ = \ \begin{cases}
      1 & \text{if $\dim \GW_d(w) = \dim X(w \cdot z_d^P)$ and 
        $w \cdot z_d^P W_P = u W_P$\,;} \\
      0 & \text{otherwise.}
    \end{cases}
  \end{split}
\end{equation}
It turns out that, if this invariant is non-zero, then $d = \alv \in
H_2(X)$ for a unique $P$-cosmall root $\al$, and we have $u W_P = w
s_\al W_P$.  A similar formula holds for the more general 2-point
$K$-theoretic Gromov-Witten invariants, see Remark~\ref{rmk:2kgw}
below.

We apply our methods to give a new proof of the (equivariant) quantum
Chevalley formula for any product of a Schubert divisor with an
arbitrary Schubert class in the equivariant quantum ring
$\QH_T(X)$.  In fact, all Gromov-Witten invariants required in such
a product can be obtained from (\ref{eqn:2gw}) combined with the
divisor axiom in Gromov-Witten theory
\cite{kontsevich.manin:gromov-witten}.  The quantum Chevalley formula
was first stated in a lecture given by Dale Peterson at M.I.T., and a
proof was later supplied by Fulton and Woodward
\cite{fulton.woodward:quantum}.  The equivariant generalization is due
to the second author \cite{mihalcea:equivariant*1} and states that all
terms of a product involving a Schubert divisor in $\QH_T(X)$ are also
visible in the equivariant cohomology $H^*_T(X)$ or in $\QH(X)$.

We remark that if $P$ is not a Borel subgroup of $G$, then the ring
$\QH_T(X)$ is not generated by divisor classes.  However, it was
demonstrated in \cite{mihalcea:equivariant*1} that all (3 point, genus
zero) equivariant Gromov-Witten invariants of $X$ can be computed with
an explicit algorithm based on the equivariant quantum Chevalley
formula.  This has been applied by Lam and Shimozono to prove that the
equivariant Gromov-Witten invariants of $X$ coincide with certain
structure constants of the equivariant homology of the affine
Grassmannian \cite{lam.shimozono:quantum}.

Our paper is organized as follows.  In section~\ref{sec:schubvars} we
recall some basic facts about Schubert classes on $X$.
Section~\ref{sec:hecke} defines the Hecke product and gives
combinatorial proofs of its main properties.  In
section~\ref{sec:comb_zd} we use the statement of Theorem~1 to give a
combinatorial construction of the element $z_d^P \in W$.  We then
prove a key technical result stating that for all effective degrees $0
\leq d' \leq d \in H_2(X)$ we have $z_{d'}^P \cdot z_{d-d'}^P W_p \leq
z_d^P W_P$.  We remark that this inequality is easy to deduce from the
geometric definition $X(z_d^P) := \Gamma_d(X(1))$, but we need to work
combinatorially to obtain a complete proof of Theorem~1.  The results
in Section~\ref{sec:comb_zd} include some basic facts concerning
cosmall roots and Hecke products of reflections that are proved using
the classification of root systems (Lemmas \ref{lem:uniquemax},
\ref{lem:cover}, and \ref{lem:replace}).  All other results in our
paper are deduced from these facts in a type independent setup.
Theorem~1 is established in section~\ref{sec:curve_nbhd}, where we
also apply this result to recover a well known formula
\cite{zak:tangents, hwang.kebekus:geometry,
  chaput.manivel.ea:quantum*1} for the minimal degree of a rational
curve between two given points in a cominuscule variety.  In
Section~\ref{sec:criteria} we prove some equivalent conditions for
$P$-cosmall roots, one of them stating that $\al$ is $P$-cosmall if
and only if $\dim X(s_\al) = \int_{C_\al} c_1(T_X) - 1$.  In
combinatorial terms this implies that $\al$ is $B$-cosmall if and only
if $\ell(s_\al) = 2 \height(\al^\vee) - 1$.  Section~\ref{sec:gw_inv}
uses these results and related inequalities to prove an explicit
formula for any two-point Gromov-Witten invariant of $X$, and
Section~\ref{sec:chevalley} proves the equivariant quantum Chevalley
formula.  While this proof logically depends on many earlier results,
including the combinatorial construction of $z_d^P$, we finish our
paper by noting that the concept of curve neighborhoods can be used to
give a very short geometric proof of the quantum Chevalley formula.

We thank Pierre-Emmanuel Chaput and Nicolas Perrin for inspiring
collaboration on related projects, and Mark Shimozono for helpful
discussions.


\section{Schubert varieties}\label{sec:schubvars}

In this section we fix our notation for Schubert varieties and state
some basic facts.  Proofs can be found in e.g.\
\cite{gonciulea.lakshmibai:flag}.  Let $X = G/P$ be a homogeneous
space defined by a connected, simply connected, semisimple complex Lie
group $G$ and a parabolic subgroup $P$.  Fix also a maximal torus $T$
and a Borel subgroup $B$ such that $T \subset B \subset P \subset G$.
Let $R$ be the associated root system, with positive roots $R^+$ and
simple roots $\Delta \subset R^+$.  Let $W = N_G(T)/T$ be the Weyl
group of $G$ and $W_P = N_P(T)/T$ the Weyl group of $P$.  The
parabolic subgroup $P$ corresponds to the set of simple roots
$\Delta_P = \{ \be \in \Delta \mid s_\be \in W_P \}$.  The group $W_P$
is generated by the simple reflections $s_\be$ for $\be \in \Delta_P$.
Set $R_P = R \cap \Z\Delta_P$ and $R^+_P = R^+ \cap \Z\Delta_P$, where
$\Z\Delta_P = \Span_\Z(\Delta_P)$ is the group spanned by $\Delta_P$.

For each element $w \in W$ we let $I(w) = R^+ \cap w^{-1}(- R^+) = \{
\al \in R^+ \mid w(\al) < 0 \}$ denote the {\em inversion set\/} of
$w$.  The second expression uses the partial order $\leq$ on $\R\Delta
= \Span_\R(\Delta)$ defined by $\al \geq \be$ if and only if $\al-\be$
is a linear combination with non-negative coefficients of the simple
roots $\Delta$.  The length of $w$ is defined by $\ell(w) = |I(w)|$.
Equivalently, $\ell(w)$ is the minimal number of simple reflections
that $w$ can be a product of.  Define the length of the coset $w W_P
\in W/W_P$ to be $\ell(w W_P) = |I(w) \ssm R^+_P|$.  The element $w$
can be written uniquely as $w = u v$ such that $I(u) \cap R^+_P =
\emptyset$ and $v \in W_P$.  We then have $u W_P = w W_P$ and $\ell(u)
= \ell(w W_P)$.  The element $u$ is called the {\em minimal
  representative\/} for the coset $w W_P$.  Similarly, if $w_P$
denotes the longest element of $W_P$, then $u w_P$ is the {\em maximal
  representative\/} for $w W_P$.  Let $W^P \subset W$ be the set of
all minimal representatives for cosets in $W/W_P$.

Let $w_0$ be the longest element in $W$ and let $B^\op = w_0 B w_0
\subset G$ be the Borel subgroup opposite to $B$.  For $w \in W$ we
define the $B$-stable Schubert variety $X(w) = \ov{B w.P} \subset X$
and the $B^\op$-stable Schubert variety $Y(w) = \ov{B^\op w.P} \subset
X$.  These varieties depend only on the coset $w W_P$ and we have
$\dim X(w) = \codim Y(w) = \ell(w W_P)$.  We also have $X(w) \cap Y(w)
= \{ w.P \}$.  The collection of points $w.P$ for $w \in W^P$ is the
set of all $T$-fixed points in $X$.

The {\em Bruhat order\/} on $W/W_P$ is defined by $u W_P \leq w W_P$
if and only if $X(u) \subset X(w)$.  This order is compatible with the
Bruhat order on $W$ in the sense that $u W_P \leq w W_P$ whenever $u
\leq w$ in $W$.  This follows because $X(u)$ is the image of $\ov{B
  u.B}$ under the projection $G/B \to X$.

Let $(-,-)$ denote the $W$-invariant inner product on $\R\Delta$.
Each root $\al \in R$ has a coroot $\alv = \frac{2\al}{(\al,\al)}$.
The coroots form the dual root system $R^\vee = \{ \alv \mid \al \in R
\}$, with basis of simple coroots $\Delta^\vee = \{ \bev \mid \be \in
\Delta \}$.  For $\be \in \Delta$ we let $\omega_\be \in \R\Delta$
denote the corresponding fundamental weight, defined by $(\omega_\be,
\alv) = \delta_{\al,\be}$ for $\al \in \Delta$.

\begin{lemma}\label{lem:salcancel}
  Let $\al \in R$ and let $S \subset R$ be any set of roots such that
  $s_\al(S) = S$.  Then we have
  \[
  \sum_{\ga\in S} (\al,\gav) = \sum_{\ga\in S} (\ga,\alv) = 0 \,.
  \]
\end{lemma}
\begin{proof}
  Since $s_\al$ is an involution of $S$ defined by $s_\al(\ga) = \ga -
  (\ga,\alv)\al$, we obtain
  \[
  \sum_{\ga\in S} (\ga,\alv) = \sum_{\ga\in S}
  \frac{\ga-s_\al(\ga)}{\al} = 0 \,.
  \]
  Since we have $s_{\alv}(S^\vee) = S^\vee$, the first sum of the
  lemma is also equal to zero.
\end{proof}

We need the following observation, which can also be found in
\cite[p.~648]{fulton.woodward:quantum}.

\begin{lemma}\label{lem:alpha_unique}
  Let $\al \in R^+ \ssm R^+_P$.  Then $\al$ is uniquely determined by
  the coset $s_\al W_P \in W/W_P$.
\end{lemma}
\begin{proof}
  Set $\la = \sum_{\be \in \Delta \ssm \Delta_P} \omega_\be$.  Then
  $W_P$ acts trivially on $\la$, while $\la - s_\al.\la = (\la,\alv)\,
  \al$ is a non-zero multiple of $\al$.  The lemma follows from this
  because distinct positive roots are never parallel.
\end{proof}

All homology and cohomology groups in this paper are taken with
integer coefficients.  Any closed irreducible subvariety $Z \subset X$
defines a fundamental homology class $[Z] \in H_{2\dim(Z)}(X)$.  We
will also use the notaion $[Z]$ for its Poincare dual class in
$H^{2\codim(Z)}(X)$.  The {\em Schubert classes\/} $[Y(w)]$ for $w \in
W^P$ form a basis for the cohomology ring $H^*(X)$.  It is convenient
to identify $H^2(X)$ with the span $\Z \{ \omega_\be \mid \be \in
\Delta\ssm\Delta_P \}$ and $H_2(X)$ with the quotient $\Z\Delta^\vee /
\Z\Delta_P^\vee$.  More precisely, for each $\be \in
\Delta\ssm\Delta_P$ we identify the class $[X(s_\be)] \in H_2(X)$ with
$\bev + \Z\Delta_P^\vee \in \Z\Delta^\vee / \Z\Delta_P^\vee$ and we
identify $[Y(s_\be)] \in H^2(X)$ with $\omega_\be$.  The Poincare
pairing $H^2(X) \otimes H_2(X) \to \Z$ is then given by the
$W$-invariant inner product $(-,-)$ on $\R\Delta$.

For each positive root $\al \in R^+ \ssm R^+_P$ there is a unique
irreducible $T$-stable curve $C_\al \subset X$ that contains $1.P$ and
$s_\al.P$.  We can restate \cite[Lemma~3.4]{fulton.woodward:quantum} as the
identity
\begin{equation}
  [C_\al] = \al^\vee + \Z\Delta_P^\vee \in H_2(X) \,.
\end{equation}
According to \cite[Lemma~3.5]{fulton.woodward:quantum} we have
\begin{equation}\label{eqn:c1T}
  c_1(T_X) \ = \ \sum_{\ga \in R^+\ssm R^+_P} \ga \ \in H^2(X) \,.
\end{equation}
Indeed, if we let $c_1 = \sum_{\ga \in R^+\ssm R^+_P} \ga$ denote the
right hand side of (\ref{eqn:c1T}), then the cited lemma implies that
$(c_1, \bev) = \int_{X(s_\be)} c_1(T_X)$ for $\be \in
\Delta\ssm\Delta_P$, and Lemma~\ref{lem:salcancel} shows that
$(c_1,\bev)=0$ for each $\be \in \Delta_P$, hence $c_1 \in
\Z\{\omega_\be \mid \be \in \Delta\ssm\Delta_P \} = H^2(X)$.

If $\la \in \Z\{\omega_\be \mid \be \in \Delta\ssm\Delta_P\}$ is an
integral weight, then the assumption that $G$ is simply connected
implies that $\la$ is represented by a character $\la : T \to \C^*$.
It therefore defines the line bundle $L_\la := G \times^P \C_{-\la} =
(G \times \C)/P$ over $X$, where $P$ acts on $G \times \C$ by $p.(g,z)
= (g p^{-1}, \la(p)^{-1}z)$.  By \cite[p.~71]{brion.kumar:frobenius}
we then have
\begin{equation}\label{eqn:c1Lla}
  c_1(L_\la) = \la \in H^2(X) \,.
\end{equation}


\section{The Hecke product}\label{sec:hecke}

Our description of curve neighborhoods is formulated in terms of the
Hecke product, which provides a monoid structure on the Weyl group
$W$.  This product describes the multiplication of basis elements in a
Hecke algebra that was first studied in the context of Tits buildings
\cite[Ch.~4, \S 2.1]{bourbaki:lie*1}.  It also describes the
composition of Demazure operators \cite{demazure:desingularisation}
and plays a key role in the combinatorial study of $K$-theory of
homogeneous spaces, see e.g.\ \cite{kostant.kumar:t-equivariant,
  fomin.kirillov:yang-baxter}.  While the Hecke product and its
properties are well known, we do not know about a reference that gives
a short unified exposition, so we have taken the opportunity to
provide one here.  Everything in this section works more generally if
$W$ is a Coxeter group, see \cite{bjorner.brenti:combinatorics} for
definitions.

For $u \in W$ and $\be \in \Delta$, define
\begin{equation}\label{eqn:dem_right}
u \cdot s_\be = \begin{cases}
  u s_\be & \text{if $u s_\be > u$;} \\
  u & \text{if $u s_\be < u$.}
\end{cases}
\end{equation}
Let $u, v \in W$ and let $v = s_{\be_1} s_{\be_2} \cdots s_{\be_\ell}$
be any reduced expression for $v$.  Define the {\em Hecke product\/}
of $u$ and $v$ by
\[
u \cdot v = u \cdot s_{\be_1} \cdot s_{\be_2} \cdot \ldots \cdot s_{\be_\ell}
\,,
\]
where the simple reflections are multiplied to $u$ in left to right
order.

We claim that this product is independent of the chosen reduced
expression for $v$.  In fact, any reduced expression for $v$ can be
obtained from any other by using finitely many {\em braid relations},
i.e.\ by steps that replace a subexpression of the form $t_0 t_1
\cdots t_{m-1}$ with $t_1 t_2 \cdots t_m$, where $t_{2i} = s_\al$ and
$t_{2i+1} = s_\be$ for given simple roots $\al, \be \in \Delta$ and
all $i \in \N$.  It is therefore enough to show that
\begin{equation}\label{eqn:dem_welldef}
  u \cdot t_0 \cdot t_1 \cdot \ldots \cdot t_{m-1} =
  u \cdot t_1 \cdot t_2 \cdot \ldots \cdot t_m \,.
\end{equation}
Let $W_{\al,\be} \subset W$ denote the parabolic subgroup generated by
$s_\al$ and $s_\be$.  Then $t_0 t_1 \cdots t_{m-1} = t_1 t_2 \cdots
t_m$ is the longest element of $W_{\al,\be}$, and both sides of
(\ref{eqn:dem_welldef}) are equal to the unique maximal representative
for the coset $u W_{\al,\be}$ in $W/W_{\al,\be}$.

Given $u, v \in W$, we will say that the product $u v$ is {\em
  reduced\/} if $\ell(u v) = \ell(u) + \ell(v)$.  This implies that $w
\cdot u v = (w \cdot u) \cdot v$ for all $w \in W$.  Notice also that
for $\be \in \Delta$ and $v \in W$ we have
\begin{equation}\label{eqn:dem_left}
s_\be \cdot v = \begin{cases}
  s_\be v & \text{if $s_\be v > v$;}\\
  v & \text{if $s_\be v < v$.}
\end{cases}
\end{equation}
In fact, if we set $v' = s_\be v$, then the identity is clear if
$\ell(v') > \ell(v)$, and otherwise $v = s_\be v'$ is a reduced
product, hence $s_\be \cdot v = (s_\be \cdot s_\be) \cdot v' = s_\be
\cdot v' = v$.

\begin{prop}\label{prop:hecke}
  Let $u, v, v', w \in W$.\smallskip

  \noin{\rm(a)} The Hecke product is associative, i.e.\ $(u \cdot
  v) \cdot w = u \cdot (v \cdot w)$.\smallskip

  \noin{\rm(b)} We have $(u \cdot v)^{-1} = v^{-1} \cdot
  u^{-1}$.\smallskip

  \noin{\rm(c)} If $v \leq v'$ then $u \cdot v \cdot w \leq u \cdot
  v' \cdot w$.\smallskip

  \noin{\rm(d)} We have $u \leq u \cdot v$, $v \leq u \cdot v$, $u v
  \leq u \cdot v$, and $\ell(u \cdot v) \leq \ell(u) +
  \ell(v)$.\smallskip
  
  \noin{\rm(e)} The element $u' = (u \cdot v) v^{-1}$ satisfies $u'
  \leq u$ and $u'v = u' \cdot v = u \cdot v$.\smallskip
  
  \noin{\rm(f)} $I(v) \subset I(u \cdot v)$.
\end{prop}
\begin{proof}
  To prove (a) it is enough to show that $(u \cdot s_\be) \cdot w = u
  \cdot (s_\be \cdot w)$ for each $\be \in \Delta$.  This is clear if
  $s_\be \cdot w = s_\be w$, and it is also clear if $u \cdot s_\be =
  u$ and $s_\be \cdot w = w$.  Assume that $u \cdot s_\be = u s_\be$
  and $s_\be \cdot w = w$, and set $w' = s_\be w$.  Since $w = s_\be
  w'$ is a reduced product, we obtain $(u \cdot s_\be) \cdot w = ((u
  \cdot s_\be) \cdot s_\be) \cdot w' = (u \cdot s_\be) \cdot w' = u
  \cdot w = u \cdot (s_\be \cdot w)$, as required.  Part (b) follows
  from the associativity together with (\ref{eqn:dem_right}) and
  (\ref{eqn:dem_left}).  To prove (c) it is enough to show that $v
  \cdot s_\be \leq v' \cdot s_\be$ for each $\be \in \Delta$.  This is
  true because $W_\be = \{1,s_\be\}$ is a parabolic subgroup of $W$,
  $v \cdot s_\be$ is the maximal representative for $v W_\be$ in
  $W/W_\be$, $v' \cdot s_\be$ is the maximal representative for $v'
  W_\be$, and $v W_\be \leq v' W_\be$.  For (d), the inequality $u
  \leq u \cdot v$ follows from (\ref{eqn:dem_right}), and $v \leq u
  \cdot v$ follows from (\ref{eqn:dem_left}).  If we write $v = v'
  s_\be$ as a reduced product with $\be \in \Delta$, then it follows
  from (c) and induction on $\ell(v)$ that $u \cdot v = (u \cdot v')
  \cdot s_\be \geq u v' \cdot s_\be \geq u v' s_\be = u v$.  The
  inequality $\ell(u \cdot v) \leq \ell(u) + \ell(v)$ is clear from
  the definition.  For (e), let $u = s_{\al_1} s_{\al_2} \cdots
  s_{\al_\ell}$ be a reduced expression for $u$, and set $y_j =
  s_{\al_j} \cdot s_{\al_{j+1}} \cdot \ldots \cdot s_{\al_\ell} \cdot
  v$ for each $j$.  Let $\{i_1 < i_2 < \dots < i_p\}$ be the set of
  indices $j$ for which $y_j \neq y_{j+1}$.  Then $u \cdot v =
  s_{\al_{i_1}} s_{\al_{i_2}} \cdots s_{\al_{i_p}} v$ is a reduced
  product, so $u' = s_{\al_{i_1}} s_{\al_{i_2}} \cdots s_{\al_{i_p}}$
  satisfies the requirements.  Finally, part (f) follows from
  (\ref{eqn:dem_left}).
\end{proof}

\begin{prop}
  Let $u, v \in W$.  The following are equivalent.\smallskip
  
  \noin{\rm(a)} The product $u v$ is reduced.\smallskip

  \noin{\rm(b)} $\ell(u \cdot v) = \ell(u) + \ell(v)$.\smallskip

  \noin{\rm(c)} $u \cdot v = u v$.\smallskip

  \noin{\rm(d)} $I(v) \subset I(u v)$.\smallskip

  \noin{\rm(e)} $I(u) \cap I(v^{-1}) = \emptyset$.
\end{prop}
\begin{proof}
  All of the implications (a) $\Rightarrow$ (b) $\Rightarrow$ (c)
  $\Rightarrow$ (d) $\Rightarrow$ (e) $\Rightarrow$ (a) follow easily
  from the definitions and Proposition~\ref{prop:hecke}.
\end{proof}

The Hecke product also defines a product $W \times W/W_P \to W/W_P$
given by 
\[
u \cdot (w W_P) = (u \cdot w) W_P \,.
\]
To see that this
is well defined, write $w = w' w''$ with $w' \in W^P$ and $w'' \in
W_P$, and set $v = u \cdot w'$ and $p = v^{-1} (v \cdot w'')$.  Then
$u \cdot w = (u \cdot w') \cdot w'' = v \cdot w'' = v p$, and since $p
\leq w''$ by Proposition~\ref{prop:hecke}(e) we must have $p \in
W_P$.  It follows that $(u \cdot w) W_P = (u \cdot w') W_P$, as
required.

Notice also that for $\be \in \Delta$ we have
\begin{equation}\label{eqn:demP_left}
  s_\be \cdot (w W_P) = \begin{cases}
    s_\be w W_P & \text{if $s_\be w W_P > w W_P$;}\\
    w W_P & \text{if $s_\be w W_P \leq w W_P$.}
  \end{cases}
\end{equation}
In fact, if $s_\be w W_P > w W_P$, then we must have $s_\be w > w$ by
compatibility of the Bruhat orders, and otherwise the inequality $w
W_P \leq (s_\be \cdot w) W_P$ implies that $s_\be w W_P = w W_P$.  The
identity follows from this.

\begin{prop}\label{prop:heckeP}
  For $u, v \in W$ we have $\ell(u \cdot v W_P) \leq \ell(u) + \ell(v
  W_P)$.  Moreover, if $\ell(u \cdot v W_P) = \ell(u) + \ell(v W_P)$
  then we must have $u \cdot v W_P = u v W_P$.
\end{prop}
\begin{proof}
  This follows from equation (\ref{eqn:demP_left}).
\end{proof}


\section{Combinatorial construction of $z_d$}\label{sec:comb_zd}

\subsection{Complete flag varieties $G/B$}

Given a degree $d \in H_2(G/B) = \Z \Delta^\vee$, the maximal elements
of the set $\{ \al \in R^+ \mid \alv \leq d \}$ are called {\em
  maximal roots\/} of $d$.  The root $\al \in R^+$ is {\em cosmall\/}
if $\al$ is a maximal root of $\alv$.  For example, this holds if
$\al$ is a simple root, a long root, or if $R$ is simply laced.
Notice also that if $\al \in R^+$ is a maximal root of any degree,
then $\al$ is automatically cosmall.

\begin{example}\label{ex:classical}
  If $R$ has one of the classical Lie types $A_{n-1}$, $B_n$, $C_n$, or
  $D_n$, then we can identify $R$ with a subset of $\R^n$ as follows.
  Let $e_1,\dots,e_n$ be the standard basis for $\R^n$ and set $\be_i =
  e_{i+1}-e_i$ for $1 \leq i \leq n-1$.  We also set $\be_0 = e_1$,
  $\wh\be_0 = 2e_1$, and $\be_{-1}=e_2+e_1$.  The following table lists
  the simple, long, short, and cosmall (positive) roots in each type.
  \medskip


\noindent
\begin{tabular}{|l|l|l|l|}
\hline\hline
$\!A_{n\!-\!1}\!\!$ 
& Simple 
& \multicolumn{2}{|l|}{$\be_1,\dots,\be_{n-1}$} 
\\\cline{2-4}
& Long
& $e_j-e_i = \be_i+\be_{i+1}+\dots+\be_{j-1}$ 
& $1 \leq i < j \leq n$ 
\\\cline{2-4}
& Cosmall 
& \multicolumn{2}{|l|}{All positive roots.}
\\
\hline\hline
$B_n$ 
& Simple 
& \multicolumn{2}{|l|}{$\be_0,\be_1,\dots,\be_{n-1}$}
\\\cline{2-4}
& Long
& $e_j-e_i= \be_i+\be_{i+1}+\dots+\be_{j-1}$ 
& $1 \leq i < j \leq n$ 
\\\cline{3-4}
& & $e_j+e_i =$ 
& $1 \leq i < j \leq n$
\\
& & $2\be_0+2\be_1+\dots+2\be_{i-1}+\be_i+\dots+\be_{j-1}$ &
\\\cline{2-4}
& Short 
& $e_i = \be_0+\be_1+\dots+\be_{i-1}$ & $1 \leq i \leq n$ 
\\\cline{2-4}
& Cosmall 
& \multicolumn{2}{|l|}{$e_1$, and $e_j-e_i$ for $1 \leq i < j
  \leq n$, and $e_j+e_i$ for $1 \leq i < j \leq n$.} 
\\
\hline\hline
&&\multicolumn{2}{|l|}{}\vspace{-3.5mm}\\
$C_n$ 
& Simple 
& \multicolumn{2}{|l|}{$\wh\be_0,\be_1,\dots,\be_{n-1}$} 
\\\cline{2-4}
&&\multicolumn{2}{|l|}{}\vspace{-3.5mm}\\
& Long 
& $2e_i = \wh\be_0 + 2\be_1 + \dots + 2\be_{i-1}$ 
& $1 \leq i \leq n$ 
\\\cline{2-4}
& Short 
& $e_j-e_i = \be_i+\be_{i+1}+\dots+\be_{j-1}$ 
& $1 \leq i
< j \leq n$ 
\\\cline{3-4}
& & $e_j+e_i =$ 
& $1 \leq i < j \leq n$ 
\\
&&\multicolumn{2}{|l|}{}\vspace{-4mm}\\
& & $\wh\be_0+2\be_1+\dots+2\be_{i-1}+\be_i+\dots+\be_{j-1}$
&
\\\cline{2-4}
& Cosmall 
& \multicolumn{2}{|l|}{$2e_i$ for $1 \leq i \leq n$, and 
$e_j-e_i$ for $1 \leq i < j \leq n$.} 
\\
\hline\hline
$D_n$ 
& Simple 
& \multicolumn{2}{|l|}{$\be_{-1},\be_1,\dots,\be_{n-1}$} 
\\\cline{2-4}
& Long 
& $e_j-e_i = \be_i+\be_{i+1}+\dots+\be_{j-1}$ 
& $1 \leq i < j \leq n$ 
\\\cline{3-4}
& & $e_j+e_1 = \be_{-1} + \be_2 + \be_3 + \dots + \be_{j-1}$ 
& $2 \leq j \leq n$ 
\\\cline{3-4}
& & $e_j+e_i =$ 
& $2 \leq i < j \leq n$ 
\\
& & $\be_{-1}+\be_1+2\be_2+\dots+2\be_{i-1}+\be_i+\dots+\be_{j-1}$ 
& 
\\\cline{2-4}
& Cosmall 
& \multicolumn{2}{|l|}{All positive roots.} 
\\
\hline\hline
\end{tabular}
\smallskip
\end{example}

\begin{example}\label{ex:cosmall_g2}
  Assume that $R$ has type $G_2$, with $\Delta = \{ \be_1, \be_2 \}$
  where $\be_2$ is the long root.  Then the cosmall roots of $R$
  consist of $\be_1$, $\be_2$, $3\be_1+\be_2$, and $3\be_1+2\be_2$.
\end{example}

\begin{example}\label{ex:cosmall_f4}
  Assume that $R$ has type $F_4$, with $\Delta =
  \{\be_1,\be_2,\be_3,\be_4\}$ and Dynkin diagram 1 --- 2 =$>$= 3 ---
  4.  Then the cosmall roots of $R$ consist of $\be_1$, $\be_2$,
  $\be_3$, $\be_4$, $\be_1+\be_2$, $\be_3+\be_4$, $\be_2+2\be_3$,
  $\be_1+\be_2+2\be_3$, $\be_1+2\be_2+2\be_3$, $\be_2+2\be_3+2\be_4$,
  $\be_1+\be_2+2\be_3+2\be_4$, $\be_1+2\be_2+2\be_3+2\be_4$,
  $\be_1+2\be_2+4\be_3+2\be_4$, $\be_1+3\be_2+4\be_3+2\be_4$, and
  $2\be_1+3\be_2+4\be_3+2\be_4$.
\end{example}

For $\al \in R^+$ we set $\supp(\al) = \{ \be \in \Delta \mid \be \leq
\al \}$.  Two positive roots $\al$ and $\be$ are {\em separated\/} if
$\supp(\al) \cup \supp(\be)$ is a disconnected subset of the Dynkin
diagram.  Equivalently, every root in the support of $\al$ is
perpendicular to every root in the support of $\be$.  Notice that if
$\al$ and $\be$ are separated roots, then $s_\al \cdot s_\be = s_\be
\cdot s_\al = s_\al s_\be$.  Given $d, d' \in \R \Delta$ we let $d
\bigvee d'$ denote the smallest element in $\R \Delta$ that is greater
than or equal to both $d$ and $d'$.

\begin{lemma}\label{lem:uniquemax}\mbox{}

  \noin{\rm(a)} For each $\al \in R^+$ there exists exactly one
  maximal root of $\alv$.\smallskip

  \noin{\rm(b)} If $\al,\be \in R^+$ are non-separated roots, then
  $\al \bigvee \be \in R^+$ is also a root.
\end{lemma}
\begin{proof}
  The lemma has been checked case by case when $R$ has exceptional Lie
  type, so we will assume that $R$ has classical type and use the
  notation of Example~\ref{ex:classical}.  Let $\al \in R^+$.  If
  $\al$ is a long root, then $\al$ is the unique maximal root of
  $\alv$, so assume that $\al$ is short.  If $R$ has type $B_n$, then
  $\al = e_i$ for some $i$.  If $i=1$, then $\al$ is the unique
  maximal root of $\alv$, and otherwise the unique maximal root of
  $\alv$ is $e_i+e_{i-1}$.  If $R$ has type $C_n$, then we have either
  $\al = e_j-e_i$ in which case $\al$ is the unique maximal root of
  $\alv$, or $\al = e_j+e_i$ in which case the unique maximal root of
  $\alv$ is $2e_j$.  This proves part (a).  Part (b) follows by
  inspection of the table in Example~\ref{ex:classical}.
\end{proof}

\begin{cor}\label{cor:maxmax}
  If $\al,\be \in R^+$ are non-separated roots and $\al$ is a maximal
  root of $\alv \bigvee \bev$, then $\be \leq \al$.
\end{cor}
\begin{proof}
  It follows from Lemma~\ref{lem:uniquemax}(b) that $\alv \bigvee
  \bev = \gav$ for some root $\ga \in R^+$, after which
  Lemma~\ref{lem:uniquemax}(a) implies that $\al$ is the only maximal
  root of $\alv \bigvee \bev$.  It follows that $\be \leq \al$.
\end{proof}

\begin{defn}
  Given an effective degree $d \in \Z \Delta^\vee$, we define an
  element $z_d \in W$ as follows.  If $d=0$ then set $z_d=1$.
  Otherwise we set $z_d = s_\al \cdot z_{d-\alv}$ where $\al$ is any
  maximal root of $d$.
\end{defn}

We prove that this is well defined by induction on $d$.  If $\al$ and
$\be$ are distinct maximal roots of $d$, then
Corollary~\ref{cor:maxmax} implies that $\al$ and $\be$ are separated.
It follows that $\al$ is a maximal root of $d-\bev$ and $\be$ is a
maximal root of $d-\alv$, so we obtain $s_\be \cdot z_{d-\bev} = s_\be
\cdot s_\al \cdot s_{d-\bev-\alv} = s_\al \cdot s_\be \cdot
s_{d-\alv-\bev} = s_\al \cdot z_{d-\alv}$, as required.

\begin{lemma}\label{lem:cover}
  Let $\al, \be \in R^+$ be cosmall roots.

  \noin{\em(a)} We have $\al \leq \be$ if and only if $\al^\vee \leq
  \be^\vee$.

  \noin{\em(b)} If $\al < \be$, then there exists a cosmall root $\ga$
  such that $\al < \ga \leq \be$ and $\gav-\alv$ is a simple coroot.
\end{lemma}
\begin{proof}
  The lemma has been checked case by case when $R$ has exceptional Lie
  type, so we will assume that $R$ has classical type.  It follows by
  inspection of Example~\ref{ex:classical} that, if $\al < \be$ is a
  covering relation in the partially ordered set of cosmall roots in
  $R^+$, i.e.\ no cosmall root is strictly in between $\al$ and $\be$,
  then $\bev-\alv$ is a simple coroot.  This proves part (b), which in
  turn shows that $\al\leq\be$ implies $\alv\leq\bev$.  On the other
  hand, if $\alv\leq\bev$, then since Lemma~\ref{lem:uniquemax}(a)
  implies that $\be$ is the unique maximal root of $\bev$, we obtain
  $\al \leq \be$.
\end{proof}

\begin{prop}\label{prop:commute}
  Let $\al, \be \in R^+$.\smallskip
  
  \noin{\rm(a)} Assume that $\be \in \Delta$.  Then we have
  $s_\al\cdot s_\be = s_\be\cdot s_\al$ if and only if $(\al,\be) \geq
  0$.\smallskip

  \noin{\rm(b)} If $\al$ is a maximal root of $\alv+\bev$, then
  $s_\al\cdot s_\be = s_\be\cdot s_\al$.
\end{prop}
\begin{proof}
  Assume that $\be$ is a simple root.  Then we have $s_\al \cdot s_\be
  = s_\be \cdot s_\al$ if and only $s_\al s_\be = s_\be s_\al$ or
  $\ell(s_\al s_\be) < \ell(s_\al)$.  The inequality $\ell(s_\al
  s_\be) < \ell(s_\al)$ is equivalent to $(\al,\be) > 0$, and the
  identity $s_\al s_\be = s_\be s_\al$ holds if and only if $\al=\be$
  or $(\al,\be)=0$.  Part (a) follows from this.

  Now let $\al,\be \in R^+$ be arbitrary positive roots such that
  $\al$ is a maximal root of $\alv+\bev$.  We must show that $s_\al
  \cdot s_\be = s_\be \cdot s_\al$.  The assumptions imply that $\al$
  is a maximal root of $\alv+\gav$ for all $\ga \in \supp(\be)$.
  Since $s_\be$ is a product of simple reflections $s_\ga$ for $\ga
  \in \supp(\be)$, we may assume that $\be$ is a simple root.

  Assume that $(\al,\be) < 0$, and set $\de = s_\be(\al)$.  Then $\de
  = \al - (\al,\bev)\be > \al$ and $\dev = \alv - (\be,\alv)\bev >
  \alv$.  Let $\ga' \geq \de$ be a maximal root of $\dev$.  Since $\al
  < \ga'$, it follows from Lemma~\ref{lem:cover} that there exists a
  cosmall root $\ga$ such that $\al < \ga \leq \ga'$ and $\gav - \alv$
  is a simple coroot.  Since we also have $\gav \leq {\ga'}^\vee \leq
  \dev$, we must have $\gav - \alv = \bev$, which contradicts that
  $\al$ is a maximal root of $\alv+\bev$.  We conclude that $(\al,\be)
  \geq 0$, so part (b) follows from part (a).
\end{proof}

Let $d \in \Z \Delta^\vee$ be an effective degree.  A {\em greedy
  decomposition\/} of $d$ is a sequence $(\al_1,\al_2,\dots,\al_k)$ of
positive roots satisfying the recursive condition that $\al_1$ is a
maximal root of $d$, and $(\al_2,\dots,\al_k)$ is a greedy
decomposition of $d - \al_1^\vee$.  The empty sequence is the only
greedy decomposition of the degree $0 \in \Z \Delta^\vee$.  If
$(\al_1,\dots,\al_k)$ is a greedy decomposition of $d$, then it
follows from the definition that $z_d = s_{\al_1} \cdot s_{\al_2}
\cdot \ldots \cdot s_{\al_k}$.  Furthermore, it follows from
Proposition~\ref{prop:commute}(b) that $s_{\al_i} \cdot s_{\al_j} =
s_{\al_j} \cdot s_{\al_i}$ for all $1 \leq i,j \leq k$.  We record the
following consequence.

\begin{cor}\label{cor:zd_inverse}
  For any effective degree $d \in \Z \Delta^\vee$ we have $(z_d)^{-1}
  = z_d$.
\end{cor}

Notice also that if $(\al_1,\dots,\al_k)$ and $(\be_1,\dots,\be_l)$
are greedy decompositions of the same degree $d$, then these
decompositions are equal up to reordering.  To see this, notice that
if $\al_1 \neq \be_1$, then Corollary~\ref{cor:maxmax} shows that that
$\be_1$ is a maximal root of $d - \al_1^\vee$ and $\al_1$ is a maximal
root of $d - \be_1^\vee$.  Let $(\ga_1,\dots,\ga_p)$ be a greedy
decomposition of $d - \al_1^\vee - \be_1^\vee$.  Then
$(\al_2,\dots,\al_k)$ and $(\be_1,\ga_1,\dots,\ga_p)$ are both greedy
decompositions of $d-\al_1^\vee$, and $(\be_2,\dots,\be_l)$ and
$(\al_1,\ga_1,\dots,\ga_p)$ are both greedy decompositions of
$d-\be_1^\vee$.  It therefore follows by induction on $d$ that all of
the sequences $(\al_1,\dots,\al_k)$,
$(\al_1,\be_1,\ga_1,\dots,\ga_p)$, and $(\be_1,\dots,\be_l)$ are
reorderings of each other.

Our proof of the following lemma relies heavily on the classification
of root systems.  It would be very interesting to find a type
independent proof.

\begin{lemma}\label{lem:replace}\mbox{}\smallskip
  
  \noin{\rm(a)} If $\al \in R^+$ is not cosmall, then $s_\al <
  z_{\alv}$.\smallskip
  
  \noin{\rm(b)} If $\al$ and $\be$ are cosmall roots, then $s_\al
  \cdot s_\be \leq z_{\alv+\bev}$.
\end{lemma}
\begin{proof}[Proof of Lemma~\ref{lem:replace}]
  The lemma has been checked case by case when $R$ has exceptional Lie
  type, so we will assume that $R$ has classical type.  We will use
  the notation of Example~\ref{ex:classical} for the roots in $R$.
  The coroots are given by $(e_j-e_i)^\vee = e_j-e_i$, $(e_j+e_i)^\vee
  = e_j+e_i$, $(e_i)^\vee = 2e_i$, and $(2e_i)^\vee = e_i$.

  The Weyl group $W$ is a subgroup of $\Aut(\R^n)$.  For $-1 \leq i
  \leq n-1$ we set $s_i = s_{\be_i} \in \Aut(\R^n)$.  Depending on the
  Lie type of $R$, these reflections may or may not be elements of
  $W$.  We also set $\sigma_k = s_{e_k}$ for $1 \leq k \leq n$, and
  set $\nu_{i,j} = s_{e_j+e_i}$, $\tau_{i,j} = s_{e_j-e_i}$, $u_{i,j}
  = s_i s_{i+1} \cdots s_{j-1}$, and $d_{i,j} = s_{j-1} s_{j-2} \cdots
  s_i$ for $1 \leq i < j \leq n$.  To reduce the number of special
  cases, we furthermore set $\tau_{i,j} = u_{i,j} = d_{i,j} = 1$ for
  $i \geq j$.  Notice that the Hecke products of these elements depend
  on the Lie type of $R$.  For example, if $R$ has type $D_n$, then
  $s_{-1}$ is a simple reflection and $s_{-1}\cdot s_{-1} = s_{-1}$,
  whereas $s_{-1}\cdot s_{-1} = s_0 s_1 s_0 s_1 = s_{-1} s_1$ if $R$
  has type $B_n$.

  Assume that $\al \in R^+$ is not cosmall.  Then the root system $R$
  is not simply laced.  If $R$ has type $B_n$, then $\al=e_i$ where $2
  \leq i \leq n$, the greedy decomposition of $\alv$ is $(e_i+e_{i-1},
  e_i-e_{i-1})$, and $s_\al = \sigma_i = s_{i-1} \cdot \sigma_{i-1}
  \cdot s_{i-1} < \sigma_{i-1} \cdot s_{i-1} \cdot \sigma_{i-1} \cdot
  s_{i-1} = \nu_{i-1,i} \cdot s_{i-1} = z_{\alv}$.  If $R$ has type
  $C_n$, then $\al = e_j+e_i$ where $1 \leq i < j \leq n$, the greedy
  decomposition of $\alv$ is $(2e_j, 2e_i)$, and $s_\al = \nu_{i,j} =
  \sigma_i \cdot \tau_{i,j} \cdot \sigma_i = \sigma_i \cdot d_{i+1,j}
  \cdot u_{i,j} \cdot \sigma_i = d_{i+1,j} \cdot \sigma_i \cdot
  u_{i,j} \cdot \sigma_i < d_{i,j} \cdot \sigma_i \cdot u_{i,j} \cdot
  \sigma_i = \sigma_j \cdot \sigma_i = z_{\alv}$.  This proves part
  (a).

  Now let $\al, \be \in R^+$ be cosmall roots.  We must prove the
  inequality $s_\al \cdot s_\be \leq z_{\alv+\bev}$ in the Bruhat
  order of $W$.  Notice that this implies that $s_\be \cdot s_\al =
  (s_\al \cdot s_\be)^{-1} \leq (z_{\alv+\bev})^{-1} = z_{\alv+\bev}$,
  hence we are free to interchange $\al$ and $\be$.  We may assume
  that $\al$ (and $\be$) is not a maximal root of $\alv+\bev$, since
  otherwise $(\al,\be)$ is a greedy decomposition and there is nothing
  to prove.  In what follows we will use the convention that zero
  vectors should be omitted in any specification of a greedy
  decomposition.  We consider five main cases.  All commutations of
  factors in the identities below follow from
  Proposition~\ref{prop:commute}.\smallskip
  
  \noin{\bf Case 1}: Assume that $\al = e_k-e_i$ and $\be = e_l-e_j$
  for some $i<k$ and $j<l$.  Up to interchanging $\al$ and $\be$, the
  assumption that $\al$ and $\be$ are not maximal roots of $\alv+\bev$
  implies that $i < j < k < l$.  In this case the greedy decomposition
  of $\alv+\bev$ is $(e_{l}-e_{i} , e_{k}-e_{j})$, and $s_\al \cdot
  s_\be \leq z_{\alv+\bev}$ holds because
  \[
  \begin{split}
    \tau_{j,l}\cdot \tau_{i,k} &=
    \tau_{j,l}\cdot d_{j,k}\cdot \tau_{i,j}\cdot u_{j,k} =
    \tau_{j,l}\cdot d_{j+1,k}\cdot \tau_{i,j}\cdot u_{j,k} =
    \tau_{j,l}\cdot \tau_{i,j}\cdot d_{j+1,k}\cdot u_{j,k} \\ &=
    \tau_{j,l}\cdot \tau_{i,j}\cdot \tau_{j,k} =
    \tau_{j,l}\cdot u_{i,j-1}\cdot d_{i,j}\cdot \tau_{j,k} =
    u_{i,j-1}\cdot \tau_{j,l}\cdot d_{i,j}\cdot \tau_{j,k} \\ &\leq
    u_{i,j}\cdot \tau_{j,l}\cdot d_{i,j}\cdot \tau_{j,k} =
    \tau_{i,l}\cdot \tau_{j,k} \,.
  \end{split}
  \]
  
  \noin{\bf Case 2}: Assume that $\al = 2e_j$ for some $j$.  Then $R$
  has type $C_n$, and since $\al$ and $\be$ are not maximal roots of
  $\alv+\bev$ we must have $\be = e_k-e_i$ where $i \leq j \leq k$.
  The greedy decomposition of $\alv+\bev$ is $(2e_{k} , e_{j}-e_{i})$;
  as mentioned above, the vector $e_j-e_i$ is omitted if $i=j$.  The
  inequality $s_\al \cdot s_\be \leq z_{\alv+\bev}$ holds because
  \[
  \begin{split}
    \sigma_{j}\cdot \tau_{i,k} &=
    d_{i,j}\cdot \sigma_{i}\cdot u_{i,j}\cdot \tau_{i,k} =
    d_{i,j}\cdot \sigma_{i}\cdot u_{i+1,j}\cdot \tau_{i,k}  =
    d_{i,j}\cdot u_{i+1,j}\cdot \sigma_{i}\cdot \tau_{i,k} \\ &=
    \tau_{i,j}\cdot \sigma_{i}\cdot \tau_{i,k} =
    \tau_{i,j}\cdot \sigma_{i}\cdot d_{i+1,k}\cdot u_{i,k} =
    \tau_{i,j}\cdot d_{i+1,k}\cdot \sigma_{i}\cdot u_{i,k} \leq \\ &=
    \tau_{i,j}\cdot d_{i,k}\cdot \sigma_{i}\cdot u_{i,k} =
    \tau_{i,j}\cdot \sigma_{k} \,.
  \end{split}
  \]
  
  \noin{\bf Case 3}: Assume that $\al = e_1$.  Then $R$ has type
  $B_n$, and since $\al$ and $\be$ are not maximal roots of
  $\alv+\bev$ we must have $\be = e_i-e_1$ for some $i>1$.  The greedy
  decomposition of $\alv+\bev$ is $(e_{i}+e_{1})$, and $s_\al \cdot
  s_\be \leq z_{\alv+\bev}$ holds because
  \[
  s_{0}\cdot \tau_{1,i} \leq
  s_{0}\cdot \tau_{1,i}\cdot s_{0} =
  \nu_{1,i} \,.
  \]
  
  \noin{\bf Case 4}: Assume that $\al = e_j+e_i$ and $\be = e_l-e_k$
  for some $i<j$ and $k<l$.  Then $R$ has type $B_n$ or $D_n$.  The
  assumption that $\al$ is not a maximal root of $\alv+\bev$ holds if
  and only if we have either $k \leq j < l$, or $i< l \leq j$ and $k
  \leq i < j-1$.  If $k \leq j < l$, then the inequality $s_\al \cdot
  s_\be \leq z_{\alv+\bev}$ follows from the subcases 4a, 4b, and 4c
  below, and if $i< l \leq j$ and $k \leq i < j-1$, then it follows
  from the subcases 4d and 4e.\smallskip

  \noin{\bf Case 4a}: Assume that $\alpha = e_{j}+e_{i}$ and $\beta =
  e_{k}-e_{i}$ where $i<j<k$.  Then the greedy decomposition of
  $\alv+\bev$ is $(e_{k}+e_{j})$.  If $R$ has type $B_n$, then the
  inequality $s_\al \cdot s_\be \leq z_{\alv+\bev}$ holds because
  \[
  \begin{split}
    \nu_{i,j}\cdot \tau_{i,k} 
    &=  \sigma_{i}\cdot \tau_{i,j}\cdot \sigma_{i}\cdot \tau_{i,k} 
    =  \sigma_{i}\cdot u_{i,j-1}\cdot d_{i,j}\cdot \sigma_{i}\cdot
    \tau_{i,k} \\
    &=  \sigma_{i}\cdot u_{i,j-1}\cdot d_{i,j}\cdot \sigma_{i}\cdot
    u_{i,j}\cdot \tau_{j,k}\cdot d_{i,j}
    =  \sigma_{i}\cdot u_{i,j-1}\cdot \sigma_{j}\cdot \tau_{j,k}\cdot d_{i,j} \\
    &=  \sigma_{i}\cdot \sigma_{j}\cdot \tau_{j,k}\cdot u_{i,j-1}\cdot d_{i,j}
    =  \sigma_{i}\cdot \sigma_{j}\cdot \tau_{j,k}\cdot \tau_{i,j} \\
    &=  \sigma_{j}\cdot \tau_{j,k}\cdot \sigma_{i}\cdot \tau_{i,j} 
    \leq  \sigma_{j}\cdot \tau_{j,k}\cdot \sigma_{j}
    =  \nu_{j,k} \hspace{10mm}\text{(by Case 2).}
  \end{split}
  \]
  Otherwise $R$ has type $D_n$ and the inequality holds because
  \[
  \begin{split}
    \nu_{i,j} &\cdot \tau_{i,k} 
    =  \nu_{i,j}\cdot u_{i,k}\cdot d_{i,k-1}
    =  d_{1,i}\cdot d_{2,j}\cdot s_{-1}\cdot u_{2,j}\cdot u_{1,i}\cdot
    u_{i,k}\cdot d_{i,k-1} \\
    &=  d_{1,i}\cdot d_{2,j}\cdot s_{-1}\cdot u_{2,j}\cdot
    u_{1,k}\cdot d_{i,k-1}
    =  d_{1,i}\cdot d_{2,j}\cdot s_{-1}\cdot u_{1,k}\cdot
    u_{1,j-1}\cdot d_{i,k-1} \\
    &=  d_{1,i}\cdot d_{2,j}\cdot s_{-1}\cdot s_{1}\cdot u_{2,k}\cdot
    u_{1,j-1}\cdot d_{i,k-1} \\
    &=  d_{1,i}\cdot d_{2,j}\cdot s_{1}\cdot s_{-1}\cdot u_{2,k}\cdot
    u_{1,j-1}\cdot d_{i,k-1} \\
    &=  d_{1,i}\cdot d_{1,j}\cdot s_{-1}\cdot u_{2,k}\cdot
    u_{1,j-1}\cdot d_{i,k-1} \\
    &=  d_{1,i}\cdot d_{1,j}\cdot s_{-1}\cdot u_{2,k}\cdot
    u_{1,j-1}\cdot d_{j,k-1}\cdot d_{i,j} \\
    &=  d_{1,i}\cdot d_{1,j}\cdot s_{-1}\cdot u_{2,k}\cdot
    d_{j,k-1}\cdot u_{1,j-1}\cdot d_{i,j} \\
    &=  d_{1,i}\cdot d_{1,j}\cdot s_{-1}\cdot u_{2,k}\cdot
    d_{j,k-1}\cdot u_{1,j}\cdot d_{i,j-1} \\
    &=  d_{1,i}\cdot d_{1,j}\cdot s_{-1}\cdot u_{2,k}\cdot
    d_{j,k-1}\cdot d_{i+1,j}\cdot u_{1,j} \\
    &=  d_{1,i}\cdot d_{1,j}\cdot s_{-1}\cdot u_{2,k}\cdot
    d_{i+1,k-1}\cdot u_{1,j}
    =  d_{1,i}\cdot d_{1,j}\cdot s_{-1}\cdot d_{i+2,k}\cdot
    u_{2,k}\cdot u_{1,j} \\
    &=  d_{1,j}\cdot d_{2,i+1}\cdot s_{-1}\cdot d_{i+2,k}\cdot
    u_{2,k}\cdot u_{1,j}
    =  d_{1,j}\cdot d_{i+2,k}\cdot d_{2,i+1}\cdot s_{-1}\cdot
    u_{2,k}\cdot u_{1,j} \\
    &\leq  d_{1,j}\cdot d_{2,k}\cdot s_{-1}\cdot u_{2,k}\cdot u_{1,j}
    =  \nu_{j,k} \,.
  \end{split}
  \]
  
  \noindent{\bf Case 4b}: Assume that $\alpha = e_{k}+e_{j}$ and
  $\beta = e_{l}-e_{i}$ where $i\leq j<k<l$.  Then the greedy
  decomposition of $\alpha^\vee+\beta^\vee$ is $(e_{l}+e_{k} ,
  e_{j}-e_{i})$, and $s_\al \cdot s_\be \leq z_{\alv+\bev}$ holds
  because
  \[
  \begin{split}
    \nu_{j,k}\cdot \tau_{i,l} 
    &=  \nu_{j,k}\cdot u_{i,j}\cdot \tau_{j,l}\cdot d_{i,j} 
    =  u_{i,j}\cdot \nu_{j,k}\cdot \tau_{j,l}\cdot d_{i,j}
    \leq \hspace{5mm}\text{(by Case 4a)}\\
    & u_{i,j}\cdot \nu_{k,l}\cdot d_{i,j}
    =  \nu_{k,l}\cdot u_{i,j}\cdot d_{i,j}
    =  \nu_{k,l}\cdot \tau_{i,j} \,.
  \end{split}
  \]
  
  \noindent{\bf Case 4c}: Assume that $\alpha = e_{k}+e_{i}$ and
  $\beta = e_{l}-e_{j}$ where $i<j\leq k<l$.  Then the greedy
  decomposition of $\alpha^\vee+\beta^\vee$ is $(e_{l}+e_{i} ,
  e_{k}-e_{j})$, and $s_\al \cdot s_\be \leq z_{\alv+\bev}$ holds
  because
  \[
  \begin{split}
    \nu_{i,k}\cdot \tau_{j,l} 
    &=  \nu_{i,k}\cdot u_{j,k}\cdot \tau_{k,l}\cdot d_{j,k}
    =  \nu_{i,k}\cdot u_{j,k}\cdot d_{k+1,l}\cdot u_{k,l}\cdot d_{j,k} \\
    &=  u_{j,k}\cdot d_{k+1,l}\cdot \nu_{i,k}\cdot u_{k,l}\cdot d_{j,k}
    \leq  u_{j,k}\cdot d_{k,l}\cdot \nu_{i,k}\cdot u_{k,l}\cdot d_{j,k} \\
    &=  u_{j,k}\cdot \nu_{i,l}\cdot d_{j,k}
    =  \nu_{i,l}\cdot u_{j,k}\cdot d_{j,k}
    =  \nu_{i,l}\cdot \tau_{j,k} \,.
  \end{split}
  \]
  
  \noindent{\bf Case 4d}: Assume that $\alpha = e_{k}+e_{j}$ and
  $\beta = e_{k}-e_{i}$ where $i\leq j<k-1$.  Then the greedy
  decomposition of $\alpha^\vee+\beta^\vee$ is $(e_{k}+e_{k-1} ,
  e_{k}-e_{k-1} , e_{j}-e_{i})$, and $s_\al \cdot s_\be \leq
  z_{\alv+\bev}$ holds because
  \[
  \begin{split}
    \nu_{j,k}\cdot \tau_{i,k} 
    &=  \nu_{j,k}\cdot u_{i,j}\cdot \tau_{j,k}\cdot d_{i,j}
    =  \nu_{j,k}\cdot u_{i,j}\cdot d_{j+1,k}\cdot u_{j,k}\cdot d_{i,j} \\
    &=  u_{i,j}\cdot d_{j+1,k}\cdot \nu_{j,k}\cdot u_{j,k}\cdot d_{i,j}
    \leq  u_{i,j}\cdot d_{j,k}\cdot \nu_{j,k}\cdot u_{j,k}\cdot d_{i,j} \\
    &=  u_{i,j}\cdot s_{k-1}\cdot d_{j,k-1}\cdot \nu_{j,k}\cdot
    u_{j,k-1}\cdot s_{k-1}\cdot d_{i,j} \\ 
    &=  u_{i,j}\cdot s_{k-1}\cdot \nu_{k-1,k}\cdot s_{k-1}\cdot d_{i,j}
    =  u_{i,j}\cdot \nu_{k-1,k}\cdot s_{k-1}\cdot s_{k-1}\cdot d_{i,j} \\
    &=  u_{i,j}\cdot \nu_{k-1,k}\cdot s_{k-1}\cdot d_{i,j}
    =  \nu_{k-1,k}\cdot s_{k-1}\cdot u_{i,j}\cdot d_{i,j}
    =  \nu_{k-1,k}\cdot s_{k-1}\cdot \tau_{i,j} \,.
  \end{split}
  \]
  
  \noindent{\bf Case 4e}: Assume that $\alpha = e_{l}+e_{j}$ and
  $\beta = e_{k}-e_{i}$ where $i\leq j<k<l$.  Then the greedy
  decomposition of $\alpha^\vee+\beta^\vee$ is $(e_{l}+e_{k} ,
  e_{j}-e_{i})$, and $s_\al \cdot s_\be \leq z_{\alv+\bev}$ holds
  because
  \[
  \begin{split}
    \nu_{j,l}\cdot \tau_{i,k} 
    &=  \nu_{j,l}\cdot u_{i,j}\cdot \tau_{j,k}\cdot d_{i,j}
    =  \nu_{j,l}\cdot u_{i,j}\cdot d_{j+1,k}\cdot u_{j,k}\cdot d_{i,j} \\
    &=  u_{i,j}\cdot d_{j+1,k}\cdot \nu_{j,l}\cdot u_{j,k}\cdot d_{i,j}
    \leq  u_{i,j}\cdot d_{j,k}\cdot \nu_{j,l}\cdot u_{j,k}\cdot d_{i,j} \\
    &=  u_{i,j}\cdot \nu_{k,l}\cdot d_{i,j}
    =  \nu_{k,l}\cdot u_{i,j}\cdot d_{i,j}
    =  \nu_{k,l}\cdot \tau_{i,j} \,.
  \end{split}
  \]
  
  \noin{\bf Case 5}: Assume that $\al = e_l+e_k$ and $\be = e_j+e_i$
  for some $i<j$ and $k<l$.  Then $R$ has type $B_n$ or $D_n$.  The
  assumption that $\al$ and $\be$ are not maximal roots of $\alv+\bev$
  implies that $i<l$ and $k<j$.  If $j \neq l$, then the inequality
  $s_\al \cdot s_\be \leq z_{\alv+\bev}$
  follows from the subcases 5a, 5b, and 5c below, and if $j = l$, then
  it follows from the subcases 5d and 5e.\smallskip
  
  \noindent{\bf Case 5a}: Assume that $\alpha = e_{l}+e_{j}$ and
  $\beta = e_{k}+e_{i}$ where $i<j<k<l$.  Then the greedy
  decomposition of $\alpha^\vee+\beta^\vee$ is $(e_{l}+e_{k} ,
  e_{j}+e_{i})$.  If $R$ has type $B_n$, then the inequality 
  $s_\al \cdot s_\be \leq z_{\alv+\bev}$ holds because
  \[
  \begin{split}
    \nu_{j,l}\cdot \nu_{i,k} 
    &=  \nu_{j,l}\cdot \sigma_{i}\cdot \tau_{i,k}\cdot \sigma_{i} 
    =  \sigma_{i}\cdot \nu_{j,l}\cdot \tau_{i,k}\cdot \sigma_{i} 
    \leq  \sigma_{i}\cdot \nu_{k,l}\cdot \tau_{i,j}\cdot \sigma_{i} 
    \hspace{3mm}\text{(by Case 4)}\\
    &=  \nu_{k,l}\cdot \sigma_{i}\cdot \tau_{i,j}\cdot \sigma_{i} 
    =  \nu_{k,l}\cdot \nu_{i,j} \,.
  \end{split}
  \]
  Otherwise $R$ has type $D_n$ and the inequality holds because
  \[
  \begin{split}
    \nu_{j,l} &\cdot \nu_{i,k} 
    =  \nu_{j,l}\cdot d_{1,i}\cdot \nu_{1,k}\cdot u_{1,i}
    =  \nu_{j,l}\cdot d_{1,i}\cdot s_{-1}\cdot \tau_{2,k}\cdot
    s_{-1}\cdot u_{1,i} \\ 
    &=  d_{1,i}\cdot s_{-1}\cdot \nu_{j,l}\cdot \tau_{2,k}\cdot
    s_{-1}\cdot u_{1,i}
    \leq  d_{1,i}\cdot s_{-1}\cdot \nu_{k,l}\cdot \tau_{2,j}\cdot
    s_{-1}\cdot u_{1,i} \hspace{2mm}\text{(by Case 4)}\\ 
    &=  \nu_{k,l}\cdot d_{1,i}\cdot s_{-1}\cdot \tau_{2,j}\cdot
    s_{-1}\cdot u_{1,i} 
    =  \nu_{k,l}\cdot d_{1,i}\cdot \nu_{1,j}\cdot u_{1,i} 
    =  \nu_{k,l}\cdot \nu_{i,j} \,.
  \end{split}
  \]
  
  \noindent{\bf Case 5b}: Assume that $\alpha = e_{l}+e_{i}$ and
  $\beta = e_{k}+e_{j}$ where $i<j<k<l$.  Then the greedy
  decomposition of $\alpha^\vee+\beta^\vee$ is $(e_{l}+e_{k} ,
  e_{j}+e_{i})$.  If $R$ has type $B_n$, then the inequality $s_\al
  \cdot s_\be \leq z_{\alv+\bev}$ holds because
  \[
  \begin{split}
    \nu_{i,l}\cdot \nu_{j,k} 
    &=  \sigma_{i}\cdot \tau_{i,l}\cdot \sigma_{i}\cdot \nu_{j,k}
    =  \sigma_{i}\cdot \tau_{i,l}\cdot \nu_{j,k}\cdot \sigma_{i}
    \leq  \sigma_{i}\cdot \nu_{k,l}\cdot \tau_{i,j}\cdot \sigma_{i} 
    \hspace{3mm}\text{(by Case 4)}\\
    &=  \nu_{k,l}\cdot \sigma_{i}\cdot \tau_{i,j}\cdot \sigma_{i}
    =  \nu_{k,l}\cdot \nu_{i,j} \,.
  \end{split}
  \]
  Otherwise $R$ has type $D_n$ and the inequality holds because
  \[
  \begin{split}
    \nu_{i,l} &\cdot \nu_{j,k} 
    =  d_{1,i}\cdot \nu_{1,l}\cdot u_{1,i}\cdot \nu_{j,k}
    =  d_{1,i}\cdot s_{-1}\cdot \tau_{2,l}\cdot s_{-1}\cdot
    u_{1,i}\cdot \nu_{j,k} \\ 
    &=  d_{1,i}\cdot s_{-1}\cdot \tau_{2,l}\cdot \nu_{j,k}\cdot
    s_{-1}\cdot u_{1,i}
    \leq  d_{1,i}\cdot s_{-1}\cdot \nu_{k,l}\cdot \tau_{2,j}\cdot
    s_{-1}\cdot u_{1,i} \hspace{2mm}\text{(by Case 4)}\\ 
    &=  \nu_{k,l}\cdot d_{1,i}\cdot s_{-1}\cdot \tau_{2,j}\cdot
    s_{-1}\cdot u_{1,i}
    =  \nu_{k,l}\cdot d_{1,i}\cdot \nu_{1,j}\cdot u_{1,i}
    =  \nu_{k,l}\cdot \nu_{i,j} \,.
  \end{split}
  \]
  
  \noindent{\bf Case 5c}: Assume that $\alpha = e_{k}+e_{i}$ and
  $\beta = e_{j}+e_{i}$ where $i<j<k$.  If $i=1$, then the assumption
  that $\al$ is not a maximal root of $\alv+\bev$ implies that $R$ has
  type $B_n$.  In this case the greedy decomposition of $\alv+\bev$ is
  $(e_{k}+e_{j} , e_{1})$, and the inequality $s_\al \cdot s_\be \leq
  z_{\alv+\bev}$ holds because
  \[
  \nu_{1,k} \cdot \nu_{1,j}
  = \nu_{1,k} \cdot s_0 \cdot \tau_{1,j} \cdot s_0
  = \nu_{1,k} \cdot \tau_{1,j} \cdot s_0
  \leq \nu_{j,k} \cdot s_0 \hspace{5mm}\text{(by Case 4).}
  \]
  Otherwise we have $1 < i < j < k$, the greedy decomposition of
  $\alpha^\vee+\beta^\vee$ is $(e_{k}+e_{j} , e_{i}+e_{i-1} ,
  e_{i}-e_{i-1})$, and the inequality follows from Case 5a because
  \[
  \nu_{i,k}\cdot \nu_{i,j} 
  =  \nu_{i,k}\cdot s_{i-1}\cdot \nu_{i-1,j}\cdot s_{i-1}
  =  \nu_{i,k}\cdot \nu_{i-1,j}\cdot s_{i-1}
  \leq  \nu_{j,k}\cdot \nu_{i-1,i}\cdot s_{i-1} \,.
  \]
  
  \noindent{\bf Case 5d}: Assume that $\alpha = e_{k}+e_{j}$ and
  $\beta = e_{k}+e_{i}$ where $i<j<k-1$.  Then the greedy
  decomposition of $\alpha^\vee+\beta^\vee$ is $(e_{k}+e_{k-1} ,
  e_{k}-e_{k-1} , e_{j}+e_{i})$, and the inequality $s_\al \cdot s_\be
  \leq z_{\alv+\bev}$ holds because
  \[
  \begin{split}
    \nu_{j,k}\cdot \nu_{i,k} 
    &=  \nu_{j,k}\cdot s_{k-1}\cdot \nu_{i,k-1}\cdot s_{k-1}
    =  \nu_{j,k}\cdot \nu_{i,k-1}\cdot s_{k-1} 
    \leq  \hspace{5mm}\text{(by Case 5a)} \\
    & \nu_{k-1,k}\cdot \nu_{i,j}\cdot s_{k-1}
    =  \nu_{k-1,k}\cdot s_{k-1}\cdot \nu_{i,j} \,.
  \end{split}
  \]
  
  \noindent{\bf Case 5e}: Assume that $\alpha = e_{j}+e_{i}$ and
  $\beta = e_{j}+e_{i}$ where $i<j-1$.  If $i=1$, then the assumption
  that $\al$ is not a maximal root of $\alv+\bev$ implies that $R$ has
  type $B_n$.  In this case the greedy decomposition of $\alv+\bev$ is
  $(e_{j}+e_{j-1} , e_{j}-e_{j-1} , e_{1})$, and the inequality $s_\al
  \cdot s_\be \leq z_{\alv+\bev}$ holds because
  \[
    \nu_{1,i}\cdot \nu_{1,i} 
    =  \nu_{1,i}\cdot s_{0}\cdot \tau_{1,i}\cdot s_{0}
    =  \nu_{1,i}\cdot \tau_{1,i}\cdot s_{0}
    \leq  \nu_{i-1,i}\cdot s_{i-1}\cdot s_{0} \hspace{5mm}\text{(by
      Case 4).}
  \]
  Otherwise we have $1<i<j-1$, the greedy decomposition of
  $\alpha^\vee+\beta^\vee$ is $(e_{j}+e_{j-1} , e_{j}-e_{j-1} ,
  e_{i}+e_{i-1} , e_{i}-e_{i-1})$, and the inequality follows from
  Case 5d because
  \[
  \nu_{i,j}\cdot \nu_{i,j} 
  = \nu_{i,j}\cdot s_{i-1}\cdot \nu_{i-1,j}\cdot s_{i-1}
  = \nu_{i,j}\cdot \nu_{i-1,j}\cdot s_{i-1}
  \leq \nu_{j-1,j}\cdot s_{j-1}\cdot \nu_{i-1,i}\cdot s_{i-1} \,.
  \]
  Since the above cases cover all possibilities for $\al$ and $\be$,
  the proof is complete.
\end{proof}

\begin{thm}\label{thm:zd_ieq}
  Let $d \in \Z \Delta^\vee$ be an effective degree.\smallskip

  \noin{\rm(a)} If $\al \in R^+$ satisfies $\alv \leq d$, then
  $s_\al \cdot z_{d-\alv} \leq z_d$.\smallskip

  \noin{\rm(b)} If $0 \leq d' \leq d$, then $z_{d'} \cdot z_{d-d'}
  \leq z_d$.
\end{thm}
\begin{proof}
  We proceed by induction on $d$, the case $d=0$ being clear.  Let $d
  > 0$ and assume that the theorem is true for all strictly smaller
  degrees.  We first show that part (b) follows from part (a).  Given
  an effective degree $d'$ with $0 < d' < d$, let $\al \in R^+$ be a
  maximal root of $d'$.  Then part (a) and the induction hypothesis
  imply that
  \[
  z_{d'}\cdot z_{d-d'} = 
  s_\al \cdot z_{d'-\alv} \cdot z_{d-d'} \leq
  s_\al \cdot z_{d-\alv} \leq z_d
  \]
  as required.

  We prove part (a) by descending induction on $\al$.  The statement
  is true by definition if $\al$ is a maximal root of $d$.  
  Assume $\al$ is not a maximal root of $d$.
  If $\al$ is not cosmall, then let $\be > \al$ be a maximal root
  of $\alv$.  We then obtain from Lemma~\ref{lem:replace}(a) that
  \[
  s_\al \cdot z_{d-\alv} \leq
  z_{\alv} \cdot z_{d-\alv} =
  s_\be \cdot z_{\alv-\bev} \cdot z_{d-\alv} \leq
  s_\be \cdot z_{d-\bev} \leq
  z_d \,.
  \]

  We may therefore assume that $\al$ is cosmall.  Since $\al$ is not a
  maximal root of $d$, we can choose a cosmall root $\ga$ such that
  $\al < \ga$ and $\gav \leq d$.  By Lemma~\ref{lem:cover} we may
  assume that $\gav-\alv$ is a simple coroot.  We can therefore choose
  a maximal root $\be$ of $d-\alv$ such that $\gav-\alv \leq \bev$.
  Now let $\ga' \geq \ga$ be a maximal root of $\alv + \bev$.  We
  finally obtain from Lemma~\ref{lem:replace}(b) that
  \begin{multline*}
  s_\al \cdot z_{d-\alv} =
  s_\al \cdot s_\be \cdot z_{d-\alv-\bev} \leq
  z_{\alv+\bev} \cdot z_{d-\alv-\bev} \\ =
  s_{\ga'} \cdot z_{\alv+\bev-{\ga'}^\vee} \cdot z_{d-\alv-\bev} \leq
  s_{\ga'} \cdot z_{d-{\ga'}^\vee} \leq
  z_d \,.
  \end{multline*}
  This completes the proof.
\end{proof}

\subsection{General homogeneous spaces}

We finish this section by extending the construction of $z_d$ to an
arbitrary homogeneous space $X = G/P$.  Given a degree $d \in H_2(X) =
\Z\Delta^\vee / \Z\Delta_P^\vee$, the maximal elements of the set $\{
\al \in R^+\ssm R^+_P \mid \alv+\Delta_P^\vee \leq d \}$ are called
{\em maximal roots\/} of $d$.  The root $\al \in R^+ \ssm R^+_P$ is
called {\em $P$-cosmall\/} if $\al$ is a maximal root of
$\alv+\Delta_P^\vee \in H_2(X)$.  Notice that any $P$-cosmall root is
cosmall, and a $B$-cosmall root is the same as a cosmall root.  The
highest root in $R$ is $P$-cosmall for every parabolic subgroup $P$.
If $\al$ is $P$-cosmall, then Proposition~\ref{prop:commute} implies
that $s_\al \cdot s_\be = s_\be \cdot s_\al$ for each $\be \in
\Delta_P$.  It follows that $s_\al \cdot w_P = w_P \cdot s_\al$.

Define a {\em greedy decomposition\/} of $d \in H_2(X)$ to be a
sequence of positive roots $(\al_1, \al_2, \dots, \al_k)$ such that
$\al_1 \in R^+ \ssm R^+_P$ is a maximal root of $d$ and
$(\al_2,\dots,\al_k)$ is a greedy decomposition of $d - \al_1^\vee \in
H_2(X)$.  The empty sequence is the only greedy decomposition of $0
\in H_2(X)$.  If $(\al_1,\al_2,\dots,\al_k)$ is a greedy decomposition
of $d \in H_2(X)$, then for any sufficiently large degree $e \in
H_2(G/B)$ such that $e + \Z\Delta_P^\vee = d$ there exist positive
roots $\ga_1,\dots,\ga_m \in R^+_P$ such that $(\al_1, \dots, \al_k,
\ga_1, \dots, \ga_m)$ is a greedy decomposition of $e$.  It follows
from this that any other greedy decomposition of $d$ is a reordering
of $(\al_1,\dots,\al_k)$.

If $d \in H_2(X)$ is any effective degree and $(\al_1,\dots,\al_k)$ is
a greedy decomposition of $d$, then we let $z_d^P \in W^P$ be the
unique element satisfying
\[
z_d^P w_P = s_{\al_1} \cdot s_{\al_2} \cdot \ldots \cdot s_{\al_k}
\cdot w_P \,.
\]
Notice that $w_P \cdot z_d^P w_P = z_d^P w_P$.

\begin{cor}\label{cor:dzP_ieq}
  Let $d \in H_2(X)$ be an effective degree.\smallskip

  \noin{\rm(a)} If $\al \in R^+$ satisfies $\alv +
  \Z\Delta_P^\vee \leq d$, then $s_\al \cdot z_{d-\alv}^P w_P \leq
  z_d^P w_P$.\smallskip

  \noin{\rm(b)} If $0 \leq d' \leq d$, then $z_{d'}^P \cdot z_{d-d'}^P
  w_P \leq z_d^P w_P$.\smallskip

  \noin{\rm(c)} If $\al \in R^+\ssm R^+_P$ is a maximal root of $d$,
  then $s_\al \cdot z_{d-\alv}^P w_P = z_d^P w_P$.\smallskip

  \noin{\rm(d)} For all sufficiently large degrees $e \in H_2(G/B)$
  such that $e + \Z\Delta_P = d \in H_2(X)$ we have $z^P_d w_P = z_e$.
\end{cor}
\begin{proof}
  Parts (a) and (b) follow from Theorem~\ref{thm:zd_ieq}, and parts
  (c) and (d) are clear from the definitions.
\end{proof}

\begin{remark}
  Let $0 \leq d \in H_2(X)$.  For any root $\al \in R^+$ such that
  $\alv \leq d \in H_2(X)$ we have (1) $z_{d'}^P w_P \cdot s_\al
  \leq z_d^P w_P$, (2) $s_\al \cdot z_{d'}^P w_P \leq z_d^P w_P$, (3)
  $z_{d'}^P \cdot s_\al W_P \leq z_d^P W_P$, and (4) $s_\al \cdot
  z_{d'}^P W_P \leq z_d^P W_P$, where $d' = d - \alv \in H_2(X)$.
  Furthermore, if $\al \in R^+ \ssm R^+_P$ is any maximal root of $d$,
  then (1), (2), (3), and (4) hold with equality.
\end{remark}


\section{Curve Neighborhoods}\label{sec:curve_nbhd}

\subsection{The main theorem}

Given an effective degree $d \in H_2(X) = \Z\Delta^\vee /
\Z\Delta_P^\vee$, let $\Mb_{0,n}(X,d)$ be the Kontsevich space of
$n$-pointed stable maps of degree $d$ to $X$, with total evaluation
map $\ev = (\ev_1,\dots,\ev_n) : \Mb_{0,n}(X,d) \to X^n$.  We have
\[
\dim \Mb_{0,n}(X,d) = \dim(X) + (c_1(T_X),d) + n-3
\]
where $c_1(T_X)$ is given by (\ref{eqn:c1T}).  Given any subvariety
$\Omega \subset X$, define the degree $d$ {\em curve neighborhood\/}
of $\Omega$ to be $\Gamma_d(\Omega) = \ev_1(\ev_2^{-1}(\Omega))$.  A
geometric argument in \cite[Cor.~3.3(a)]{buch.chaput.ea:finiteness}
shows that, if $Z$ is a $B$-stable Schubert variety, then so is
$\Gamma_d(\Omega)$.  The following result gives a more combinatorial
proof of this fact, and also identifies the Weyl group element of the
curve neighborhood.  This result is equivalent to Theorem~1.

\begin{thm}\label{thm:main}
  For any $w \in W$ we have $\Gamma_d(X(w)) = X(w \cdot z_d^P)$.
\end{thm}
\begin{proof}
  We prove the result by induction on $d$, the case $d=0$ being clear.
  Assume that $d>0$ and that $\Gamma_{d'}(X(u)) = X(u \cdot z_{d'}^P)$
  for all $d' < d$ and all $u \in W$.  Let $\al \in R^+ \ssm R^+_P$ be
  a maximal root of $d$ and set $v = (w \cdot s_\al)s_\al$.  Since
  $v.P \in X(w)$ and $v.C_\al$ is a curve of degree $\alv +
  \Z\Delta^\vee_P$ from $v.P$ to $(w \cdot s_\al).P$, it follows that
  $X(w \cdot s_\al) \subset \Gamma_{\alv}(X(w))$.  We therefore obtain
  $X(w \cdot z_d^P) = X(w \cdot s_\al \cdot z_{d-\alv}^P) =
  \Gamma_{d-\alv}(X(w \cdot s_\al)) \subset
  \Gamma_{d-\alv}(\Gamma_{\alv}(X(w))) \subset \Gamma_d(X(w))$.

  On the other hand, let $u.P \in \Gamma_d(X(w))$ be any $T$-fixed
  point.  Since the locus of curves of degree $d$ from $X(w)$ to $u.P$
  is a closed $T$-stable subvariety of $\Mb_{0,2}(X,d)$, it follows
  that this locus contains a $T$-stable curve $C$ connecting $u.P$ to
  a point $v.P \in X(w)$ where $v \in W^P$.  This curve must be a
  connected union of irreducible $T$-stable components.  At least one
  component contains $v.P$, and any such component has the form
  $v.C_\al$ with $\al \in R^+\ssm R^+_P$.  Choose a component
  $v.C_\al$ such that $C' = \ov{C \ssm v.C_\al}$ connects $v s_\al.P$
  to $u.P$.  Since $v s_\al.P \in X(w \cdot s_\al)$ and $[C'] \leq
  d-\alv \in H_2(X)$, it follows from Corollary~\ref{cor:dzP_ieq}(a)
  that $u.P \in \Gamma_{d-\alv}(X(w \cdot s_\al)) = X(w \cdot s_\al
  \cdot z_{d-\alv}^P) \subset X(w \cdot z_d^P)$.  Since
  $\Gamma_d(X(w))$ is $B$-stable and all its $T$-fixed points belong
  to $X(w \cdot z_d^P)$, we deduce that $\Gamma_d(X(w)) \subset X(w
  \cdot z_d^P)$.  This completes the proof.
\end{proof}

\begin{remark}
  The curve neighborhood of the opposite Schubert variety $Y(w)$ is
  given by $\Gamma_d(Y(w)) = Y(w_0 (w_0 w \cdot z_d^P))$, where $w_0
  (w_0 w \cdot z_d^P)$ may be regarded as a `decreasing' Hecke product
  of $w$ and $z_d^P$.
\end{remark}

\subsection{The moment graph}

The element $z_d^P \in W^P$ can also be constructed using the {\em
  moment graph\/} of $X$.  The vertices of this graph are the
$T$-fixed points $X^T$ and the edges are the irreducible $T$-stable
curves in $X$.  More precisely, there is an edge between $u.P$ and
$w.P$ if and only if $w W_P = u s_\al W_P$ for some root $\al \in R^+
\ssm R^+_P$; the corresponding $T$-stable curve is $u.C_\al \subset X$
which has degree $[C_\al] = \alv + \Z\Delta_P \in H_2(X)$.  Define
the {\em weight\/} of a path in the moment graph to be the sum of the
degrees of its edges.  Given $d \in H_2(X)$, let $Z_d \subset X^T$
be the subset of points $w.P$ for which there exists a path from $1.P$
to $w.P$ of weight at most $d$.  Then Theorem~1 implies that $z_d^P.P$
is the unique maximal element of $Z_d$ in the Bruhat order on $X^T$,
defined by $u.P \leq w.P$ if and only if $u W_P \leq w W_P$.

\begin{example}\label{ex:moment_B2}
  Let $X = \SO(5)/B = \OF(5)$ be the variety of isotropic flags in the
  vector space $\C^5$ equipped with an orthogonal form.  The
  corresponding root system has type $B_2$.  Let $\Delta = \{ \be_1,
  \be_2 \}$ be the simple roots, with $\be_1$ long and $\be_2$ short.
  The remaining positive roots are $\al = \be_1+\be_2$ and $\ga =
  \be_1+2\be_2$, with coroots $\alv = 2\be_1^\vee + \be_2^\vee$ and
  $\gav = \be_1^\vee + \be_2^\vee$.  Write $s_i = s_{\be_i}$ for
  $i=1,2$.  The moment graph of $X$ is displayed below, with each edge
  labeled by its degree.  Since the paths of weight $\gav$ starting at
  $1$ are $1 \to s_1 \to s_1 s_2$, $1 \to s_2 \to s_2 s_1$, and $1 \to
  s_\ga = s_2 s_1 s_2$, we have $z_{\gav} = s_\ga$.  On the other
  hand, the paths of weight $\alv$ starting at $1$ include $1 \to s_1
  \to w_0$, so $z_{\alv} = w_0 \neq s_\al = s_1 s_2 s_1$.
  \[
  {
    \psfrag{a}{\raisebox{1mm}{$\be_1$}}
    \psfrag{n}{$\be_2$}
    \psfrag{c}{\raisebox{1mm}{$\al$}}
    \psfrag{e}{\raisebox{1mm}{$\ga$}}
    \raisebox{20mm}{\pic{.8}{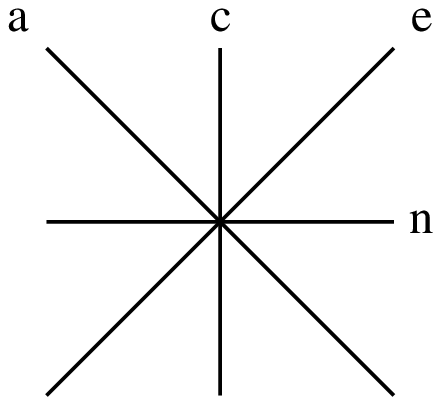}}
  }
  \ \ \ \ \ \ \ \  
  {
    \psfrag{a}{$\!\!\!\be_1^\vee$}
    \psfrag{c}{$\!\!\!\be_2^\vee$}
    \psfrag{e}{\raisebox{-1mm}{$\!\!\!\alv$}}
    \psfrag{n}{\raisebox{-.5mm}{$\!\!\!\gav$}}
    \psfrag{w0}{$\!\!s_1 s_2$}
    \psfrag{w1}{$\!\!s_1 s_2 s_1$}
    \psfrag{w2}{$\!w_0$}
    \psfrag{w3}{$\!\!\!\!\!\!\!\!\!\!\!s_2 s_1 s_2$}
    \psfrag{w4}{$\!\!\!\!\!s_2 s_1$}
    \psfrag{w5}{$\!s_2$}
    \psfrag{w6}{$1$}
    \psfrag{w7}{$s_1$}
    \pic{.8}{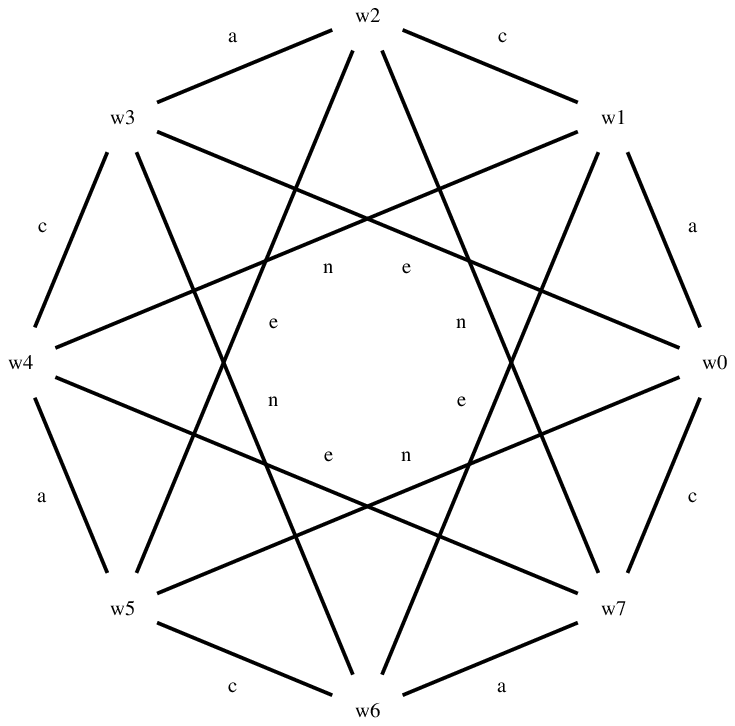}
  }
  \]
\end{example}

\subsection{Line neighborhoods}

Let $\ga \in \Delta \ssm \Delta_P$ and consider the Dynkin diagram of
the simple roots in the set $\Delta_P \cup \{\ga\}$.  We will say that
$\ga$ is a {\em Fano root\/} of $X$ if $\ga$ is at least as long as
all other roots in its connected component in this diagram.  It has
been proved by Strickland \cite{strickland:lines} and by Landsberg and
Manivel \cite{landsberg.manivel:projective*1} that $\ga$ is a Fano root
for $X$ if and only if the Fano variety of lines in $X$ of degree
$\gav \in H_2(X)$ is a homogeneous space.  In this case we will
compute the element $z_\gav^P$ that describes neighborhoods defined by
lines of degree $\gav$.

\begin{lemma}\label{lem:wP_max_root}
  Let $\ga \in \Delta\ssm \Delta_P$.  Then $w_P(\ga)$ is the largest
  root in $R \cap (\ga + \Z\Delta_P)$.
\end{lemma}
\begin{proof}
  Let $\rho$ be any maximal root in the set $R \cap (\ga +
  \Z\Delta_P)$.  Then $w_P(\rho) - \ga$ is a non-negative linear
  combination of $\Delta_P$.  Since $w_P(\Delta_P) \subset R^-$ we
  obtain $\rho - w_P(\ga) = w_P(w_P(\rho)-\ga) \leq 0$.  Since
  $w_P(\ga) \in R \cap (\ga + \Z\Delta_P)$, this implies that $\rho =
  w_P(\ga)$.
\end{proof}

\begin{prop}\label{prop:fano_nbhd}
  Let $\ga \in \Delta \ssm \Delta_P$ be any Fano root for $X$.  Then
  $w_P(\ga)$ is the unique maximal root of $\gav + \Z\Delta_P^\vee \in
  H_2(X)$.
\end{prop}
\begin{proof}
  Let $\al \in R^+\ssm R^+_P$ be any maximal root of $\gav +
  \Z\Delta_P^\vee$.  Then we have $\alv = \gav + y$ for some $y \in
  \Z\Delta_P^\vee$, and $\al = \frac{|\al|^2}{|\ga|^2}\,\ga +
  \frac{|\al|^2}{2}\,y$.  Since $\ga$ is a Fano root of $X$ and the
  support of $\al$ is contained in the connected component of $\ga$ in
  the Dynkin diagram of $\Delta_P\cup\{\ga\}$, it follows that $|\al|
  \leq |\ga|$.  We deduce that $\al \in R \cap (\ga + \Z\Delta_P)$, so
  Lemma~\ref{lem:wP_max_root} implies that $\al \leq \rho :=
  w_P(\ga)$.  Finally, since $\rho^\vee + \Z\Delta_P^\vee = \ga^\vee +
  \Z\Delta_P^\vee \in H_2(X)$, we must have $\al = \rho$.
\end{proof}

\begin{cor}
  If $\ga \in \Delta \ssm \Delta_P$ is a Fano root for $X$, then 
  $z_\gav^P W_P = w_P s_\ga W_P$.
\end{cor}
\begin{proof}
  The root $\rho = w_P(\ga)$ is the unique maximal root of $\gav \in
  H_2(X)$, so we have $z_\gav^P w_P = s_\rho \cdot w_P = w_P s_\ga w_P
  \cdot w_P = w_P s_\ga \cdot w_P$.
\end{proof}

\subsection{Degree distances in cominuscule varieties}

A simple root $\ga \in \Delta$ is {\em cominuscule\/} if, when the
highest root in $R^+$ is written as a linear combination of simple
roots, the coefficient of $\ga$ is one.  The variety $X = G/P$ is
cominuscule if $P$ is a maximal parabolic subgroup and $\Delta_P =
\Delta \ssm \{\ga\}$ for a cominuscule root $\ga$.  In this case
$H_2(X) \cong \Z$ is generated by $\gav + \Z\Delta_P^\vee$, so any
degree can be identified with an integer.  The following result is
known from combining lemmas 4.2 and 4.4 in
\cite{buch.chaput.ea:finiteness}.

\begin{thm}\label{thm:zdP_comin}
  Assume that $X = G/P$ is cominuscule with $\Delta_P = \Delta \ssm
  \{\ga\}$.  For any effective degree $d \in H_2(X)$ we have $z_d^P
  W_P = (w_P s_\ga) \cdot (w_P s_\ga) \cdot \ldots \cdot (w_P s_\ga)
  W_P$ where the Hecke product has $d$ factors.
\end{thm}
\begin{proof}
  Lemma~\ref{lem:wP_max_root} and the cominuscule condition imply that
  $\rho = w_P(\ga)$ is the highest root in $R^+$.  Since $\rho^\vee +
  \Z\Delta_P^\vee = \gamma^\vee + \Z\Delta_P^\vee$ is the smallest
  positive degree in $H_2(X)$, it follows that $\rho$ is the only
  $P$-cosmall root in $R^+\ssm R^+_P$.  The greedy decomposition of $d
  \in H_2(X)$ is therefore $(\rho,\rho,\dots,\rho)$, with $\rho$
  repeated $d$ times.  Finally, since $s_\rho \cdot w_P = w_P s_\ga
  \cdot w_P$, we obtain
  \[
  z_d^P w_P = s_\rho \cdot s_\rho \cdot \ldots \cdot s_\rho \cdot w_P
  = (w_P s_\ga) \cdot (w_P s_\ga) \cdot \ldots \cdot (w_P s_\ga) \cdot
  w_P \,,
  \]
  as required.
\end{proof}

Given two point $x,y \in X$ in a cominuscule variety, define the {\em
  degree distance\/} $d(x,y)$ to be the smallest possible degree of a
rational curve from $x$ to $y$.  The study of this integer was
suggested by Zak \cite{zak:tangents} and plays a fundamental role in
Chaput, Manivel, and Perrin's generalization
\cite{chaput.manivel.ea:quantum*1} of the `quantum equals classical'
principle from \cite{buch.kresch.ea:gromov-witten}.  The maximal value
of $d(x,y)$ was computed by Hwang and Kebekus
\cite{hwang.kebekus:geometry}.  Notice that we may choose $g \in G$
such that $g.x = 1.P$ and $g.y = u.P$ for some $u \in W^P$, in which
case $d(x,y) = d(1.P, u.P)$.  The function $d(x,y)$ is therefore
determined by the following corollary, which was obtained earlier in
\cite[Prop.~18]{chaput.manivel.ea:quantum*1}.

\begin{cor}
  Assume that $X = G/P$ is cominuscule with $\Delta_P = \Delta \ssm
  \{\ga\}$, and let $u \in W^P$.  Then $d(1.P, u.P)$ is the number of
  occurrences of $s_\ga$ in any reduced expression for $u$.
\end{cor}
\begin{proof}
  For any degree $d \in H_2(X)$ it follows from
  Theorem~\ref{thm:zdP_comin} that $u.P \in \Gamma_d(X(1))$ if and
  only if $u$ has a reduced expression with at most $d$ occurrences of
  $s_\ga$.  Set $d = d(1.P,u.P)$.  Since we have $u.P \in
  \Gamma_d(X(1))$ and $u.P \notin \Gamma_{d-1}(X(1))$, we deduce that
  $u$ has a reduced expression with exactly $d$ occurrences of
  $s_\ga$.  We finally appeal to Stembridge's result
  \cite{stembridge:fully} that all elements of $W^P$ are {\em fully
    commutative}, i.e.\ any reduced expression for an element of $W^P$
  can be obtained from any other by interchanging commuting simple
  transpositions (see also \cite[\S 2]{buch.samuel:k-theory} for an
  alternative proof of this fact).  This implies that all reduced
  expressions for $u$ have the same number of occurrences of $s_\ga$.
\end{proof}


\section{Criteria for cosmall roots}\label{sec:criteria}

In this section we prove several criteria for cosmall roots.  These
results will be useful for proving the quantum Chevalley formula.  We
start with the following two theorems which are proved together.

\begin{thm}\label{thm:Pcosmall}
  Given any root $\al \in R^+ \ssm R^+_P$ we have $\ell(s_\al W_P)
  \leq (c_1(T_X),\alv) - 1$.  Moreover, the following are
  equivalent.\smallskip
  
  \noin{\rm(a)} The root $\al$ is $P$-cosmall.\smallskip
  
  \noin{\rm(b)} We have equality $\ell(s_\al W_P) = (c_1(T_X),\alv) -
  1$.\smallskip

  \noin{\rm(c)} We have $(R^+ \ssm R^+_P) \cap s_\al(R^+_P)
  = \emptyset$ and $(\ga,\alv)=1$ for all $\ga \in I(s_\al) \ssm
  (R^+_P \cup \{\al\})$.
\end{thm}

\begin{thm}\label{thm:zdP_length}
  Let $0 < d \in H_2(X)$.  Then $\ell(z_d^P) \leq (c_1(T_X),d)-1$.
  Furthermore, if $\ell(z_d^P) = (c_1(T_X),d)-1$, then $d = \alv +
  \Z\Delta_P^\vee$ for a unique $P$-cosmall root $\al$.
\end{thm}

\begin{proof}[Proof of Theorems \ref{thm:Pcosmall} and \ref{thm:zdP_length}]
  Given $\al \in R^+ \ssm R^+_P$ we consider the sets $A = I(s_\al)
  \ssm (R^+_P \cup \{\al\})$, $B = (R^+ \ssm R^+_P) \cap
  s_\al(R^+_P)$, and $C = (R^+ \ssm R^+_P) \cap s_\al(R^+ \ssm
  R^+_P)$.  Since $R^+ \ssm R^+_P$ is the disjoint union of $\{\al\}$,
  $A$, $B$, and $C$, we obtain from (\ref{eqn:c1T}) that
  \[
  (c_1(T_X),\alv) = 2 + \sum_{\ga \in A} (\ga,\alv) + \sum_{\ga \in B}
  (\ga,\alv) + \sum_{\ga \in C} (\ga,\alv) \,.
  \]
  Notice that $|A| = \ell(s_\al W_P)-1$, $(\ga,\alv) \geq 1$ for all
  $\ga \in A \cup B$, and Lemma~\ref{lem:salcancel} implies that
  $\sum_{\ga\in C} (\ga,\alv) = 0$.  The inequality $\ell(s_\al W_P)
  \leq (c_1(T_X),\alv)-1$ and the equivalence of (b) and (c) in
  Theorem~\ref{thm:Pcosmall} follow from this.
  
  We next prove Theorem~\ref{thm:zdP_length}.  Let $0 < d \in H_2(X)$
  and let $\ga$ be a maximal root of $d$.  If $d-\gav \neq 0 \in
  H_2(X)$, then we obtain by induction that $\ell(z_d^P) =
  \ell(z_{d-\gav}^P \cdot s_\ga W_P) \leq \ell(z_{d-\gav}^P) +
  \ell(s_\ga W_P) \leq (c_1(T_X),d-\gav)-1 + (c_1(T_X),\gav)-1 =
  (c_1(T_X),d)-2$.  We deduce that, if $\ell(z_d^P) = (c_1(T_X),d)-1$,
  then $d = \gav + \Z\Delta_P^\vee$ in $H_2(X)$.  The uniqueness of
  $\ga$ follows because the greedy decomposition $(\ga)$ of $d$ is
  unique.

  Assume that $\al \in R^+ \ssm R^+_P$ satisfies $\ell(s_\al W_P) =
  (c_1(T_X),\alv) - 1$.  Since $s_\al W_P \leq z_{\alv}^P W_P$, we
  deduce from Theorem~\ref{thm:zdP_length} that $\alv = \gav \in
  H_2(X)$ for a $P$-cosmall root $\ga \in R^+ \ssm R^+_P$, and we have
  $s_\al W_P = z_{\alv}^P W_P = s_\ga W_P$.
  Lemma~\ref{lem:alpha_unique} therefore implies that $\al = \ga$.
  This proves the implication (b) $\Rightarrow$ (a) in
  Theorem~\ref{thm:Pcosmall}.

  On the other hand, assume that condition (c) fails.  If there exists
  a root $\ga \in A$ with $(\ga,\alv) \geq 2$, then $\al$ is short,
  $\ga$ is long, and $(\al,\gav) = 1$.  Set $\be = - s_\al(\ga) \in
  R^+$.  Then $\bev + \gav = \gav - s_\al(\gav) = \alv$, and
  $(\al,\bev) = (\al,\alv-\gav) = 1$.  Since $|\be| = |\ga|$ we obtain
  $s_\be(\al) + s_\ga(\al) = 2\al - (\be+\ga) = 2\al -
  \frac{|\ga|^2}{|\al|^2} \al \leq 0$.  Up to interchanging $\be$ and
  $\ga$, this implies that $s_\ga(\al) = \al-\ga < 0$, so $\al < \ga$
  and $\gav \leq \alv \in H_2(X)$.  This shows that $\al$ is not
  $P$-cosmall.  Otherwise $B \neq \emptyset$ and we can choose $\be
  \in R^+_P$ such that $s_\al(\be) \in R^+ \ssm R^+_P$.  Since the
  support of $s_\al(\be)$ is not contained in the support of $\be$, we
  must have $(\be,\alv) < 0$.  Set $\ga = s_\be(\al) = \al -
  (\al,\bev)\be > \al$.  Then $\gav = s_\be(\alv) = \alv -
  (\be,\alv)\bev = \alv \in H_2(X)$, hence $\al$ is not $P$-cosmall.
  This establishes the implication (a) $\Rightarrow$ (c) and completes
  the proof.
\end{proof}

\begin{remark}
  Theorem~\ref{thm:Pcosmall} implies that $\ell(s_\al) = 2
  \min(\height(\al),\height(\alv)) - 1$ for each $\al \in R^+$, where
  $\height(\al)$ is the sum of the coefficients obtained when $\al$ is
  expressed in the basis of simple roots.  In fact, we have
  $(c_1(T_{G/B}),\alv) = 2 \height(\alv)$, and either $\al$ or $\alv$
  is cosmall.
\end{remark}

The following criterion has so far only been proved for $B$-cosmall
roots, but we believe that it holds for $P$-cosmall roots when the
root system $R$ is simply laced.

\begin{prop}
  Let $\al \in R^+$.  Then $\al$ is cosmall if and only if $z_\alv =
  s_\al$.
\end{prop}
\begin{proof}
  This follows from Lemma~\ref{lem:replace}(a) and the definition of
  $z_{\alv}$.
\end{proof}

\begin{conj}
  Assume that $R$ is simply laced and let $\al \in R^+ \ssm R^+_P$.
  Then $\al$ is $P$-cosmall if and only if $z_{\alv}^P W_P = s_\al
  W_P$.
\end{conj}

\begin{example}\label{ex:counter}
  Let $G$ be a group of type $B_2$, let $\Delta = \{\be_1, \be_2\}$,
  $\al = \be_1+\be_2$, and $\ga = \be_1+2\be_2$ be as in
  Example~\ref{ex:moment_B2}, and let $P \subset G$ be the parabolic
  subgroup defined by $\Delta_P = \{ \be_2 \}$.  Then $s_\al W_P = s_1
  s_2 s_1 W_P = w_0 W_P$ is the longest coset, hence $s_\al W_P =
  z_\alv^P W_P$.  However, the greedy decomposition of $\alv +
  \Z\Delta_P^\vee$ is $(\ga,\ga)$, so $\al$ is not $P$-cosmall.  In
  fact, Example~\ref{ex:moment_B2} shows that $\al$ is not even
  $B$-cosmall.
\end{example}

The following definition together with
Proposition~\ref{prop:small_cosmall} below is our reason for choosing
the name `{\em cosmall\/}' in section~\ref{sec:comb_zd}.

\begin{defn}
  A positive root $\al \in R^+$ is {\em large\/} if it is long and can
  be written as the sum of two short positive roots.  Otherwise $\al$
  is {\em small}.
\end{defn}

\begin{prop}\label{prop:small_cosmall}
  Let $\al \in R^+$.  Then $\al$ is cosmall if and only if the coroot
  $\alv$ is a small root of the dual root system $R^\vee$.
\end{prop}
\begin{proof}
  By Theorem~\ref{thm:Pcosmall} it is enough to show that $\alv$ is
  large if and only if there exists a root $\ga \in I(s_\al) \ssm
  \{\al\}$ for which $(\ga,\alv) \geq 2$.  If $\alv$ is large, then
  $\al$ is short and we can write $\alv = \bev + \gav$ where $\be$ and
  $\ga$ are long positive roots.  Since we have $(\al,\bev)\leq 1$,
  $(\al,\gav)\leq 1$, and $(\al,\bev) + (\al,\gav) = (\al,\alv) = 2$,
  we deduce that $(\al,\gav)=1$ and $(\ga,\alv) \geq 2$.  Finally,
  since $s_\al(\gav) = \gav - \alv = -\bev$, it follows that $\ga \in
  I(s_\al) \ssm \{\al\}$.  Conversely, if $\ga \in I(s_\al) \ssm
  \{\al\}$ satisfies $(\ga,\alv) \geq 2$, then set $\be = -s_\al(\ga)
  \in R^+$.  Then $\al$ is short, $\ga$ and $\be$ are long, and $\bev
  = -s_\al(\gav) = \alv-\gav$.  This shows that $\alv$ is large.
\end{proof}


\section{Gromov-Witten invariants}\label{sec:gw_inv}

Given $w \in W$ and $d \in H_2(X)$, define the Gromov-Witten
variety
\[
\GW_d(w) = \ev_2^{-1}(X(w)) \ \subset \Mb_{0,2}(X,d) \,.
\]
We have $\dim \GW_d(w) = \ell(w W_P) + (c_1(T_X),d) - 1$, and
Theorem~1 implies that $\ev_1(\GW_d(w)) = \Gamma_d(X(w)) = X(w \cdot
z_d^P)$.  We need the following consequence of
\cite[Prop.~3.3]{buch.chaput.ea:finiteness}.

\begin{prop}[\cite{buch.chaput.ea:finiteness}]\label{prop:loctriv}
  The variety $\GW_d(w)$ is unirational, and the evaluation map $\ev_1
  : \GW_d(w) \to \Gamma_d(X(w))$ is a locally trivial fibration over
  the open $B$-orbit in $\Gamma_d(X(w))$.  In particular, the general
  fibers of $\ev_1$ are unirational.
\end{prop}

\begin{thm}\label{thm:push}
  Let $w \in W^P$ and $0 < d \in H_2(X)$.  Then $(\ev_1)_*
  [\GW_d(w)]$ is non-zero in $H^*(X)$ if and only if we have $d =
  \alv + \Z\Delta_P^\vee$ for some root $\al \in R^+$ such that
  $\ell(w s_\al W_P) = \ell(w) + (c_1(T_X),\alv)-1$.  In this case $\al$
  is $P$-cosmall and $(\ev_1)_* [\GW_d(w)] = [X(w s_\al)]$.
\end{thm}
\begin{proof} 
  Assume that $(\ev_1)_*[\GW_d(w)] \neq 0$.  Then the inequalities
  $\dim X(w \cdot z_d^P) \leq \ell(w)+\ell(z_d^P) \leq \ell(w) +
  (c_1(T_X),d)-1 = \dim \GW_d(w)$ must be equalities, and
  Theorem~\ref{thm:zdP_length} implies that $d = \alv+\Z\Delta_P^\vee
  \in H_2(X)$ for some $P$-cosmall root $\al$.  Since $\ell(w \cdot
  z_d^P W_P) = \ell(w)+\ell(z_d^P W_P)$ and $z_d^P W_P = s_\al W_P$,
  we obtain from Proposition~\ref{prop:heckeP} that $\Gamma_d(X(w)) =
  X(w \cdot s_\al) = X(w s_\al)$.  Finally,
  Proposition~\ref{prop:loctriv} implies that $(\ev_1)_*[\GW_d(w)] =
  [\Gamma_d(X(w))]$.  On the other hand, if $\al \in R^+$ satisfies $d
  = \alv \in H_2(X)$ and $\ell(w s_\al W_P) = \ell(w) +
  (c_1(T_X),d)-1$, then the inequalities $\ell(w s_\al W_P) \leq
  \ell(w) + \ell(s_\al W_P) \leq \ell(w) + (c_1(T_X),d)-1$ must be
  equalities, so Theorem~\ref{thm:Pcosmall} implies that $\al$ is
  $P$-cosmall.  Similarly, the inequalities $\ell(w s_\al W_P) \leq
  \dim X(w \cdot z_d^P) \leq \dim \GW_d(w)$ are equalities, hence
  $(\ev_1)_* [\GW_d(w)] \neq 0$.
\end{proof}

Let $H^*_T(X)$ denote the $T$-equivariant cohomology ring of $X$.
Each $T$-stable closed subvariety $Z \subset X$ defines an equivariant
class $[Z] \in H^*_T(X)$.  Pullback along the structure morphism $X
\to \{\pt\}$ gives $H^*_T(X)$ the structure of an algebra over the
ring $\La := H^*_T(\pt)$, and $H^*_T(X)$ is a free $\La$-module with
basis $\{ [Y(w)] : w \in W^P \}$.  For any class $\Omega \in H_T^*(X)$
we let $\int_X \Omega \in \La$ denote the proper pushforward of
$\Omega$ along the structure morphism $X \to \{\pt\}$.  The Kontsevich
space $\Mb_{0,n}(X,d)$ has a natural $T$-action given by $(t.f)(y) =
t.f(y)$ for any stable map $f : C \to X$ and $t \in T$, and the
evaluation maps $\ev_i : \Mb_{0,n}(X,d) \to X$ are $T$-equivariant.
Given classes $\Omega_1, \dots, \Omega_n \in H^*_T(X)$ and $d \in
H_2(X)$, the associated {\em equivariant Gromov-Witten invariant\/} is
defined by
\[
I_d(\Omega_1,\dots,\Omega_n) = \int_{\Mb_{0,n}(X,d)} \ev_1^*(\Omega_1)
\cdot \ev_2^*(\Omega_2) \cdots \ev_n^*(\Omega_n) \ \in \La \,.
\]
Notice that Theorem~\ref{thm:push} holds for the equivariant class
$(\ev_1)_* [\GW_d(w)] \in H_T^*(X)$, with the same proof.

\begin{cor}\label{cor:2gw}
  Let $w, u \in W^P$ and $0 < d \in H_2(X)$.  The two-point
  Gromov-Witten invariant $I_d([Y(u)], [X(w)])$ is non-zero if and
  only if there exists a root $\al \in R^+ \ssm R^+_P$ such that $d =
  \alv + \Z \Delta_P^\vee$, $\ell(w s_\al W_P) = \ell(w) +
  (c_1(T_X),\alv)-1$, and $u W_P = w s_\al W_P$.  In this case $\al$
  is $P$-cosmall and $I_d([Y(u)], [X(w)]) = 1$.
\end{cor}
\begin{proof}
  Since we have $I_d([Y(u)],[X(w)]) = \int_X [Y(u)] \cdot
  (\ev_1)_*[\GW_d(w)]$ by the projection formula, the corollary
  follows from Theorem~\ref{thm:push}.
\end{proof}

The divisor axiom \cite{kontsevich.manin:gromov-witten} (see also
\cite[(40)]{fulton.pandharipande:notes}) is valid for equivariant
Gromov-Witten invariants.  Let $Z \subset X$ be any $T$-stable divisor
and $0 < d \in H_2(X)$, and consider the variety $\ev_n^{-1}(Z)
\subset \Mb_{0,n}(X,d)$ and the morphism $\phi : \ev_n^{-1}(Z) \to
\Mb_{0,n-1}(X,d)$ that discards the $n$-th marked point in the domain
of its argument.  For a general stable map $f : C \to X$ in
$\Mb_{0,n-1}(X,d)$ we can identify the fiber $\phi^{-1}(f)$ with
$f^{-1}(Z) \subset C$, so it follows from Kleiman's transversality
theorem \cite{kleiman:transversality} that $\# \phi^{-1}(f) = ([Z],d)
= \int_d [Z]$ for all points $f$ in a dense open subset of
$\Mb_{0,n-1}(X,d)$.  We deduce that $\phi_*[\ev_n^{-1}(Z)] = ([Z],d)
\cdot [\Mb_{0,n-1}(X,d)]$, so the projection formula implies that
\begin{equation}\label{eqn:divisor_axiom}
  I_d(\Omega_1,\dots,\Omega_{n-1},[Z]) = 
  ([Z],d) \cdot I_d(\Omega_1,\dots,\Omega_{n-1}) \ \in \Lambda
\end{equation}
for all classes $\Omega_1,\dots,\Omega_{n-1} \in H_T^*(X)$.
In particular, the equivariant Gromov-Witten invariant 
$I_d(\Omega_1,\dots,\Omega_{n-1},[Z])$ depends only on the class of
$Z$ in the ordinary cohomology ring $H^*(X)$ and not on its
equivariant class in $H^*_T(X)$.

\begin{cor}\label{cor:chev_gw}
  Let $u,w \in W^P$, $\be \in \Delta$, and $0 < d \in H_2(X)$.  If the
  Gromov-Witten invariant $I_d([Y(u)],[Y(s_\be)],[X(w)])$ is non-zero,
  then there exists a unique root $\al \in R^+\ssm R^+_P$ such that
  {\rm(i)} $d = \alv + \Z\Delta_P^\vee$, {\rm(ii)} $w W_P = u s_\al
  W_P$, and {\rm(iii)} $\ell(w W_P) = \ell(u W_P) + 1 -
  (c_1(T_X),\alv)$.  If $\al \in R^+ \ssm R^+_P$ is any root
  satisfying {\rm(i)}, {\rm(ii)}, and {\rm(iii)}, then we have
  $\gw{[Y(u)],[Y(s_\be)],[X(w)]}{d} = (\omega_\be,\alv) \in \Z$.
\end{cor}
\begin{proof}
  If $I_d([Y(u)],[X(w)],[Y(s_\be)]) \neq 0$ then it follows from
  (\ref{eqn:divisor_axiom}) and Corollary~\ref{cor:2gw} that there
  exists a root $\ga \in R^+\ssm R^+_P$ such $d = \gav +
  \Z\Delta_P^\vee$, $u W_P = w s_\ga W_P$, and $\ell(w W_P) = \ell(u
  W_P) + 1 - (c_1(T_X),\gav)$.  Since $u^{-1} w s_\ga \in W_P$ we
  deduce that $\al := u^{-1} w(-\ga) \in R^+\ssm R^+_P$ satisfies (i),
  (ii), and (iii).  The uniqueness of $\al$ follows from
  Lemma~\ref{lem:alpha_unique}.
\end{proof}

\begin{remark}\label{rmk:2kgw}
  The $K$-theoretic two-point invariants are easier to compute, since
  Proposition~\ref{prop:loctriv} together with the $K$-theoretic Gysin
  formula of \cite[Thm.~3.1]{buch.mihalcea:quantum} imply that
  $(\ev_1)_*[\cO_{\GW_d(w)}] = [\cO_{X(w \cdot z_d^P)}] \in K_T(X)$ for
  any degree $d \geq 0$.  It follows that any equivariant
  $K$-theoretic two-point Gromov-Witten invariant of $X$ is given by
  \[
  \begin{split}
    &\euler{\Mb_{0,2}(X,d)}(\ev_1^* [\cO_{Y(u)}] \cdot \ev_2^*
    [\cO_{X(w)}]) 
    \ = \ \euler{X}( [\cO_{Y(u)}] \cdot [\cO_{X(w \cdot z_d^P)}] ) \ \ \\
    & \ \ \ \ \ \ = \ \euler{X}(\cO_{Y(u) \cap X(w\cdot z_d^P)})
    \ = \ \begin{cases}
      1 & \text{if $u W_P \leq w \cdot z_d^P W_P$\,;} \\
      0 & \text{otherwise.}
    \end{cases}
  \end{split}
  \]
  We refer to \cite[\S 4]{buch.mihalcea:quantum} for notation.
  Unfortunately, the $K$-theoretic invariants do not satisfy a divisor
  axiom, so this formula does not reveal any 3-point invariants.
\end{remark}


\section{The equivariant quantum Chevalley formula}
\label{sec:chevalley}

The {\em $T$-equivariant quantum cohomology ring\/} $\QH_T(X)$ is an
algebra over the polynomial ring $\La[q] := \La[q_\be : \be \in \Delta
\ssm \Delta_P]$, where $\La = H^*_T(\pt)$, which as a $\La[q]$-module
is defined by $\QH_T(X) = H^*(X) \otimes_\Z \La[q]$.  The
multiplicative structure of $\QH_T(X)$ is given by
\[
[Y(u)] \star [Y(v)] = \sum_{w,d} I_d([Y(u)], [Y(v)], [X(w)])\, q^d\,
[Y(w)] \,,
\]
where the sum is over $w \in W^P$ and $0 \leq d \in H_2(X)$, and we
write $q^d = \prod_\be q_\be^{(\omega_\be,d)}$.  It was proved in
\cite{mihalcea:equivariant, mihalcea:equivariant*1} that if $v =
s_\be$ is a simple reflection, then the product $[Y(u)] \star
[Y(s_\be)]$ contains no {\em mixed\/} terms, i.e.\ if $d \neq 0$ then
the coefficient of $q^d [Y(w)]$ is always an integer.  This fact is
also a consequence of Corollary~\ref{cor:chev_gw}.

To state the equivariant quantum Chevalley formula, we need some
notation.  Since $G$ is simply connected, each integral weight $\la
\in \Z \{ \omega_\be \mid \be \in \Delta \}$ can be identified with a
character $\la : T \to \C^*$.  Let $\C_\la$ be the corresponding
one-dimensional representation of $T$, defined by $t.z = \la(t) z$.
This representation can be viewed as a $T$-equivariant vector bundle
over a point, so it defines the equivariant Chern class $c_T(\la) :=
c_1^T(\C_\la) \in \La$.  This class should {\em not\/} be confused
with the class that $\la$ might represent in $H^2(X) = \Z \{
\omega_\be \mid \be \in \Delta \ssm \Delta_P \}$ by the notation of
section~\ref{sec:schubvars}.  The ring $\La$ is the polynomial ring
over $\Z$ generated by the classes $c_T(\omega_\be)$ for $\be \in
\Delta$.  The equivariant quantum Chevalley formula is the following
result \cite{chevalley:decompositions, fulton.woodward:quantum,
  mihalcea:equivariant*1}.

\begin{thm}\label{thm:qchevalley}
  Let $u \in W^P$ and $\be \in \Delta\ssm\Delta_P$.  Then we have
  \[\begin{split}
    &[Y(u)] \star [Y(s_\be)] = \\
    &\sum_\alpha (\omega_\beta,\alpha^\vee)\, [Y(us_\alpha)] \ + \ 
    c_T(\omega_\beta - u.\omega_\beta)\,[Y(u)] \ + \ 
    \sum_\al (\omega_\be,\alv) \, q^{\alv} [Y(u s_\al)] \ ;
    \end{split}
    \]
    the first sum is over $\al \in R^+ \ssm R^+_P$ such that $\ell(u
    s_\al W_P) = \ell(u W_P) + 1$, and the second sum is over $\al \in
    R^+\ssm R^+_P$ such that $\ell(u s_\al W_P) = \ell(u W_P) + 1 -
    (c_1(T_X),\alv)$.
\end{thm}
\begin{proof}
  It follows from Corollary~\ref{cor:chev_gw} that the second sum
  accounts for all terms with non-zero $q$-degrees.  The remaining
  terms come from the equivariant product $[Y(u)] \cdot [Y(s_\be)] \in
  H^*_T(X)$, and the coefficient of $[Y(w)]$ in this product is
  \[
  c^w_{u,s_\be} = \int_X [Y(u)] \cdot [Y(s_\be)] \cdot [X(w)] = \int_X
  [Y(u) \cap X(w)] \cdot [Y(s_\be)] \ \in \La \,.
  \]
  This coefficient is non-zero only if $u \leq w$ and $\ell(w W_P)
  \leq \ell(u W_P) + 1$.  If $\ell(w W_P) = \ell(u W_P) + 1$, then the
  intersection $Y(u)\cap X(w)$ is a one-dimensional closed $T$-stable
  subvariety of $X$ whose $T$-fixed points consist of $u.P$ and $w.P$.
  It follows that $Y(u) \cap X(w) = u.C_\al$ and $w.P = u s_\al.P$ for
  some root $\al \in R^+ \ssm R^+_P$, and we have $c^w_{u,s_\be} =
  ([Y(s_\be)],[C_\al]) = (\omega_\be,\alv)$ as claimed.  This argument
  can also be found in e.g.\ \cite[Lemma~8.1]{fulton.woodward:quantum}
  or \cite[Prop.~1.4.3]{brion:lectures}.
  
  The last remaining term is $c^u_{u,s_\be} [Y(u)]$.  The projection
  formula implies that
  \[
  c^u_{u,s_\be} = \int_X [Y(s_\be)] \cdot [u.P] = [Y(s_\be)]_{u.P}
  \]
  where $[Y(s_\be)]_{u.P} \in \La$ is the restriction of $[Y(s_\be)]$
  to the $T$-fixed point $u.P \in X$.  Set $\la = \omega_\be$ and
  notice that $L_\la = G \times^P \C_{-\la}$ is a $G$-equivariant line
  bundle with action defined by $g' . [g,z] = [g'g,z]$.  According to
  the Borel-Weil theorem \cite[p.~99]{brion.kumar:frobenius} there
  exists a $B^\op$-stable section $\sigma \in H^0(X,L_\la)$, unique up
  to scalar, and we have $\C \sigma \cong \C_{-\la}$ as a
  $T$-representation.  This implies that $\sigma : X \times \C_{-\la}
  \to L_\la$ is a morphism of $T$-equivariant line bundles.  Since the
  zero section $Z(\sigma)$ is $B^\op$-stable, we deduce from
  (\ref{eqn:c1Lla}) that $Z(\sigma) = Y(s_\be)$.  The $T$-equivariant
  class of $Y(s_\be)$ is therefore given by
  \[
  [Y(s_\be)] = [Z(\sigma)] = c_1^T(L_\la) - c_1^T(X \times \C_{-\la}) \in
  H^2_T(X) \,.
  \]
  Since the fiber of $L_\la$ over $u.P$ is $L_\la(u.P) \cong
  \C_{-u.\la}$, we obtain
  \[
  [Y(s_\be)]_{u.P} = c_1^T(\C_{-u.\la}) - c_1^T(\C_{-\la}) =
  c_T(\omega_\be - u.\omega_\be) \,.
  \]
  This finishes the proof.
\end{proof}

\begin{remark}
  The full strength of Theorem~1 is not necessary to prove the quantum
  Chevalley formula.  We here sketch a short alternative argument that
  bypasses the combinatorial construction of $z_d^P$.  It follows from
  Proposition~\ref{prop:loctriv} that $\Gamma_d(X(w))$ is a Schubert
  variety in $X$ for all $w \in W^P$ and $d \in H_2(X)$.  Using this
  we can show that, if $d > 0$, then there {\em exists\/} a positive
  root $\al \in R^+ \ssm R^+_P$ such that $\Gamma_d(X(w)) =
  \Gamma_{d-\alv}(X(w \cdot s_\al))$.  In fact, if we write
  $\Gamma_d(X(w)) = X(u)$, let $C$ be a $T$-stable curve from $v.P \in
  X(w)$ to $u.P$, and let $v.C_\al \subset C$ a component such that
  $\ov{C \ssm v.C_\al}$ connects $vs_\al.P$ to $u.P$, then we must
  have $X(u) \subset \Gamma_{d-\alv}(X(v s_\al)) \subset
  \Gamma_{d-\alv}(X(w \cdot s_\al)) \subset
  \Gamma_{d-\alv}(\Gamma_\alv(X(w))) \subset \Gamma_d(X(w))$.  The
  arguments proving Theorems \ref{thm:zdP_length} and \ref{thm:push}
  now show that $\dim \Gamma_d(X(w)) \leq \ell(w) + (c_1(T_X),d)-1$,
  with equality if and only if $d = \alv+\Z\Delta_P^\vee$ for some
  root $\al \in R^+$ such that $\ell(w s_\al W_P) = \ell(w) +
  (c_1(T_X),d)-1$.  Theorem~\ref{thm:push} and the quantum Chevalley
  formula follow from this.
\end{remark}


\providecommand{\bysame}{\leavevmode\hbox to3em{\hrulefill}\thinspace}
\providecommand{\MR}{\relax\ifhmode\unskip\space\fi MR }
\providecommand{\MRhref}[2]{%
  \href{http://www.ams.org/mathscinet-getitem?mr=#1}{#2}
}
\providecommand{\href}[2]{#2}

\end{document}